\documentclass[12pt, reqno]{amsart}

\makeatletter
\g@addto@macro{\endabstract}{\@setabstract}
\makeatother

\usepackage{graphics, stackrel}
\usepackage{amsmath, amssymb, amsthm}
\usepackage{graphicx}
\usepackage{verbatim}
\usepackage{amsfonts}
\usepackage{filecontents}
\usepackage{natbib}
\usepackage{fancyvrb}
\usepackage{color}

\usepackage[citecolor=blue, colorlinks=true, linkcolor=blue]{hyperref}

\usepackage{mathrsfs}
\usepackage{bbm}

\usepackage[left=1.2in, right=1.2in, top=1.0in, bottom=1.15in, includehead, includefoot]{geometry}

\usepackage{multirow}

\renewcommand{\leq}{\leqslant}
\renewcommand{\geq}{\geqslant}

\newcommand{\setntn}[2]{ \{ #1 : #2 \} }

\newcommand{\iidsim}{\stackrel {\textrm{ {\sc iid }}} {\sim} }

\newcommand*\diff{\mathop{}\!\mathrm{d}}
\renewcommand{\epsilon}{\varepsilon}

\newcommand{\gG}{\mathcal G}
\newcommand{\vV}{\mathcal V}
\newcommand{\hH}{\mathcal H}

\newcommand{\XX}{\mathsf X}
\renewcommand{\AA}{\mathsf A}
\newcommand{\FF}{\mathsf F}

\newcommand{\ZZ}{\mathsf Z}
\newcommand{\YY}{\mathsf Y}

\newcommand{\GG}{\mathbbm G}
\newcommand{\HH}{\mathbbm H}
\newcommand{\VV}{\mathbbm V}

\newcommand{\RR}{\mathbbm R}

\newcommand{\NN}{\mathbbm N}

\newcommand{\EE}{\mathbbm E \,}

\theoremstyle{plain}
\newtheorem{theorem}{Theorem}[section]

\newtheorem{lemma}[theorem]{Lemma}
\newtheorem{proposition}[theorem]{Proposition}

\theoremstyle{definition}

\newtheorem{example}{Example}[section]

\newtheorem{assumption}{Assumption}[section]

\begin{document}

\title{}

\date{\today}

\begin{center}
	\Large
	Dynamic Programming Deconstructed: Transformations of the Bellman Equation and Computational Efficiency\footnote{We 
		thank Fedor Iskhakov, Takashi Kamihigashi, Larry Liu and Daisuke Oyama
		for valuable feedback and suggestions, as well as audience members
		at the Econometric Society meeting in Auckland in 2018 and the 2nd
		Conference on Structural Dynamic Models in Copenhagen in 2018.
		Financial support from ARC Discovery Grant
		DP120100321 is gratefully acknowledged. \\ 
		\emph{Email addresses:} \texttt{qingyin.ma@cueb.edu.cn}, \texttt{john.stachurski@anu.edu.au} }

	\bigskip
	\normalsize
	Qingyin Ma\textsuperscript{a} and John Stachurski\textsuperscript{b} \par \bigskip
	
	\textsuperscript{a}ISEM, Capital University of Economics and Business    \\
	\textsuperscript{b}Research School of Economics, Australian National University \bigskip
	
	\today
\end{center}

\begin{abstract} 
    Some approaches to solving challenging dynamic programming problems, such as Q-learning, begin by transforming the Bellman equation into an alternative functional equation, in order to open up a new line of attack.  Our paper studies this idea systematically, with a focus on boosting computational efficiency.  We provide a characterization of the set of valid transformations of the Bellman equation, where validity means that the transformed Bellman equation maintains the link to optimality held by the original Bellman equation.  We then examine the solutions of the transformed Bellman equations and analyze correspondingly transformed versions of the algorithms used to solve for optimal policies. These investigations yield new approaches to a variety of discrete time dynamic programming problems, including those with features such as recursive preferences or desire for robustness.  Increased computational efficiency is demonstrated via time complexity arguments and numerical experiments.
    \vspace{1em}
    
    \noindent \textit{JEL Classifications:} C61, E00 \\
    \textit{Keywords:} Dynamic programming, optimality, computational efficiency
\end{abstract}

\section{Introduction}

Dynamic programming is central to the analysis of intertemporal planning
problems in management, operations research, economics, finance and other
related disciplines (see, e.g., \cite{bertsekas2017dynamic}). When combined
with statistical learning, dynamic programming also drives a number of
strikingly powerful algorithms in artificial intelligence and automated
decision systems (\cite{kochenderfer2015decision}).  At the heart of dynamic
programming lies the Bellman equation, a functional equation
that summarizes the trade off between current and future rewards for a
particular dynamic program. Under standard regularity conditions, solving the
Bellman equation allows the researcher to assign optimal values to states and
then compute optimal policies via Bellman's principle of optimality
(\cite{bellman1957dynamic}).

Despite its many successes, practical implementation of dynamic programming
algorithms is often challenging, due to high dimensionality, irregularities
and/or lack of data (see, e.g., \cite{rust1996numerical}).  A popular way to
mitigate these problems is to try to change the angle of attack by rearranging
the Bellman equation into an alternative functional form.  One example of this
is Q-learning, where the first step is to transform the Bellman
equation into a new functional equation, the solution of which is the
so-called Q-factor (see, e.g., \cite{kochenderfer2015decision} or
\cite{bertsekas2017dynamic}, Section~6.6.1).  Working with the Q-factor turns
out to be convenient when data on transition probabilities is scarce.

Other examples of transformations of the Bellman equation can be found in
\cite{kristensen18}, who investigate three versions of the Bellman equation
associated with a fixed dynamic programming problem, corresponding to the
value function, ``expected value function'' and ``integrated value
function'' respectively.  Similarly, \cite{bertsekas2017dynamic} discusses
modifying the Bellman equation by integrating out ``uncontrollable states''
(Section~6.1.5).   Each of these transformations of the Bellman equation
creates new methods for solving for the optimal policy, since
the transformations applied to the Bellman equation can be likewise
applied to the iterative techniques used to solve the
Bellman equation (e.g., value function iteration or policy iteration).

The purpose of this paper is twofold.  First, we provide the first systematic
investigation of these transformations, by developing a framework sufficiently
general that all transformations of which we are aware can be expressed as
special cases.  We then examine the solutions of the transformed Bellman
equations, along with correspondingly transformed versions of solution
algorithms such as value function iteration or policy iteration.  We provide a
condition under which the transformed Bellman equations satisfy a version of
Bellman's principle of optimality, and the correspondingly transformed
versions of the algorithms lead to optimal policies.  This puts existing
transformations on a firm theoretical foundation, in terms of their link to
optimal policies, thereby eliminating the need to check optimality or
convergence of algorithms on a case-by-case basis.

Second, on a more practical level, we use the framework developed above to
create new transformations, or to apply existing transformations in new
contexts, and to examine the convergence properties of the associated
algorithms.  Although there are many motivations for applying transformations
to the Bellman equation (such as to facilitate learning, as in the case of
Q-learning, or to simplify estimation as in \cite{rust1987optimal}), the
motivation that we focus on here is computational efficiency.  In particular,
we examine transformations of the Bellman equation that retain the links to
optimality discussed above while reducing the dimension of the state space.
Because the cost of high dimensionality is exponential, even small reductions
in the number of variables can deliver speed gains of one or two orders of
magnitude.  These claims are verified theoretically and through numerical
experiments when we consider applications.

One of the reasons that we are able to create new transformations is that we
embed our theoretical results in a very general abstract dynamic programming
framework, as found for example in \cite{bertsekas2013abstract}.  This allows
us to examine transformations of the Bellman equation in settings beyond the
traditional additively separable case, which have been developed to better
isolate risk preferences or accommodate notions of ambiguity. Examples of
such objective functions can be found in \cite{iyengar2005robust},
\cite{hansen2008robustness}, \cite{ruszczynski2010risk} and
\cite{bauerle2018stochastic}.

When analyzing solution methods and computational efficiency, we show that
successive iterates of the transformed Bellman operator converge at the same
rate as the original Bellman operator in the following sense: the $n$-th
iterate of the Bellman operator can alternatively be produced by iterating
the same number of times with the transformed Bellman operator 
and vice versa.  This means that, to judge the relative
efficiency, one only needs to consider the computational cost of a single
iteration of the transformed Bellman operator versus the original Bellman
operator.  

When treating algorithms for computing optimal policies, we focus in
particular on so-called optimistic policy iteration, which contains policy
iteration and value function iteration as special cases, and which can
typically be tuned to converge faster than either one in specific
applications.  We show that, under a combination of stability and
monotonicity conditions, the sequence of policies generated by
a transformed version of optimistic policy iteration, associated with a
particular transformation of the Bellman equation, converges to optimality.

Some results related to ours and not previously mentioned can be found in
\cite{rust1994structural}, which discusses the connection between the fixed
point of a transformed Bellman equation and optimality of the policy that
results from choosing the maximal action at each state evaluated according to
this fixed point.  This is again a special case of what we cover, specific to
one specialized class of dynamic programming problems, with discrete choices
and additively separable preferences, and refers to one specific
transformation associated with the expected value function.  

There is, of course, a very large literature on maximizing computational
efficiency in solution methods for dynamic programming problems (see, e.g.,
\cite{powell2007approximate}).  The results presented here are, in many ways, complementary to
this existing literature.  For example, fitted value iteration is a popular
technique for solving the Bellman equation via successive approximations,
combined with a function approximation step (see, e.g.,
\cite{munos2008finite}).  This methodology could also, in principle, be
applied to transformed versions of the Bellman equation and the successive
approximation techniques for solving them examined in this paper.

The rest of our paper is structured as follows.  Section~\ref{s:gf} formulates
the problem.  Section~\ref{s:opt} discusses optimality and algorithms.
Section~\ref{s:aii} gives a range of applications. Longer proofs are deferred to
the appendix.

\section{General Formulation}

\label{s:gf}

\label{s:cv_approach}

This section presents an abstract dynamic programming problem and the key concepts and operators related to the formulation. 
%related to plan factorization. 

\subsection{Problem Statement}

We use $\NN$ to denote the set of natural numbers and 
$\NN_0 := \{ 0\} \cup \NN$, while $\RR^E$ is the set of real-valued functions 
defined on some set $E$. In what follows, a dynamic programming problem consists of
\begin{itemize}
	\item a nonempty set $\XX$ called the \textit{state space},
	\item a nonempty set $\AA$ called the \textit{action space},
	\item a nonempty correspondence $\Gamma$ from $\XX$ to $\AA$ called 
	the \textit{feasible correspondence}, along with the associated set of
	\emph{state-action pairs}
	\begin{equation*}
	\FF := \setntn{(x, a) \in \XX \times \AA}{a \in \Gamma(x)},
	\end{equation*}
	\item a subset $\vV$ of $\RR^{\XX}$ called the set of \textit{candidate value 
		functions}, and
	\item a \textit{state-action aggregator} $H$ mapping $\FF \times \vV$ to
	$\RR \cup \{ -\infty \}$.
\end{itemize}

The interpretation of $H( x, a, v)$ is the
lifetime value associated with choosing action $a$ at current state $x$ and
then continuing with a reward function $v$ attributing value to states.  In
other words, $H(x, a, v)$ is an abstract representation of the value to be
maximized on the right hand side of the Bellman equation (see \eqref{eq:be} below). This abstract
representation accommodates both additively separable and nonseparable preferences 
and is based on \cite{bertsekas2013abstract}. 
The sets $\XX$ and $\AA$ are, at this point, arbitrary. They can, for example, be finite or
subsets of Euclidean vector space. 

We associate to this abstract dynamic program the Bellman equation
\begin{equation}
\label{eq:be}
v(x) 
= \sup_{a \in \Gamma (x)} H(x, a, v)
\quad \text{for all } \, x \in \XX.
\end{equation}
Stating that $v \in \vV$ solves the Bellman equation is equivalent to stating
that $v$ is a fixed point of the Bellman operator, which we denote by $T$ and
define by 
\begin{equation}
\label{eq:bop}
T \, v (x) = \sup_{a \in \Gamma (x)} H(x, a, v)
\qquad \; (x \in \XX, \; v \in \vV).
\end{equation}

\begin{example}
	\label{eg:fmdp}
	In a traditional infinite horizon finite state Markov decision process, an
	agent observes a state $x$ in finite set $\XX$ and responds with action
	$a$ from $\Gamma(x)$, a subset of finite set $\AA$.  The state then
	updates to $x'$ next period with probability $p(x, a, x')$. The objective
	is to choose a policy $\sigma \colon \XX \to \AA$ such that, when
	$\{a_t\}$ is selected according to $a_t = \sigma(x_t)$ at each $t$, the
	objective $\EE \sum_{t \geq 0} \beta^t r(x_t, a_t)$
	is maximized.  Here $\beta \in (0, 1)$ is a discount factor and $r$ is
	a real-valued reward
	function.  This fits directly into our framework if we set $\vV$ to be
	all functions from $\XX$ to $\RR$ and 
	\begin{equation}
	\label{eq:h}
	H(x, a, v) = r(x, a) + \beta \sum_{x' \in \XX} v(x') p(x, a, x').
	\end{equation}
	In particular, inserting the right hand side of \eqref{eq:h} into \eqref{eq:be}
	reproduces the standard Bellman equation for this problem (see, e.g., \cite{bertsekas2013abstract}, Example~1.2.2).
\end{example}

More detailed examples, involving those with nonseparable preferences, are given in Section \ref{ss:ex} and Section \ref{s:aii} below.

\subsection{Policies and Assumptions}

\label{ss:paa}

Let $\Sigma$ denote the set of \textit{feasible policies}, which we define as all
$\sigma \colon \XX \to \AA$ satisfying $\sigma(x) \in \Gamma (x)$ for all $x \in \XX$ and 
the regularity condition
\begin{equation}
\label{eq:f}
v \in \vV
\text{ and } 
w(x) = H (x, \sigma(x), v) \text{ on } \XX
\implies
w \in \vV.
\end{equation}
Given $v \in \vV$, a feasible policy $\sigma$ with the property that
\begin{equation} 
\label{eq:exgree}
H (x, \sigma(x), v )
= \sup_{a \in \Gamma(x)} H( x, a , v )  
\quad \text{for all $x \in \XX$}
\end{equation}
is called \emph{$v$-greedy}.
In other words, a $v$-greedy policy is one we obtain by treating $v$ as the
value function and maximizing the right hand side of the Bellman equation.

%\begin{assumption}
%    \label{a:maxex}
%    At least one $v$-greedy policy exists in $\Sigma$ for each $v$ in $\VV$.
%\end{assumption}

\begin{assumption}
	\label{a:maxex}
	There exists a subset $\VV$ of $\vV$ such that $T$ maps elements of $\VV$ into itself, and a $v$-greedy policy exists in $\Sigma$ for each $v$ in $\VV$.
\end{assumption}

Assumption~\ref{a:maxex} guarantees the existence of stationary policies and
allows us to work with maximal decision rules instead of suprema. It
is obviously satisfied for any dynamic program where the set of actions is
finite, such as Example~\ref{eg:fmdp} (by letting $\VV = \vV$), and extends to infinite choice problems when primitives are
sufficiently continuous and the choice set at each action is compact.

Given $\sigma \in \Sigma$, any function $v_\sigma$ in $\vV$ satisfying
\begin{equation*}
v_\sigma (x) = H (x, \sigma(x), v_\sigma )
\quad \text{for all } \, x \in \XX
\end{equation*}
is called a \textit{$\sigma$-value function}. We understand $v_\sigma (x)$ as
the lifetime value of following policy $\sigma$ now and forever, starting from
current state $x$.  

\begin{assumption}
	\label{a:ess}
	For each $\sigma \in \Sigma$, there is exactly one $\sigma$-value function 
	in $\vV$, denoted by $v_{\sigma}$.   
\end{assumption}

Assumption~\ref{a:ess} is required for our optimization problem to be well defined, and is
easy to verify for regular Markov decision problems such as the one described
in Example~\ref{eg:fmdp} (see, e.g., \cite{bertsekas2013abstract},
Example~1.2.1).  When rewards are unbounded or
dynamic preferences are specified recursively, the restrictions on primitives
used to establish this assumption vary substantially across applications (see,
e.g., \cite{marinacci2010unique}, \cite{bauerle2018stochastic}, or
\cite{bloise2018convex} for sufficient conditions in these contexts, and
\cite{bertsekas2013abstract} for general discussion).

\subsection{Decompositions and Plan Factorizations}

Next we introduce plan factorizations (which correspond to transformations of
the Bellman equation) in a relatively abstract form,
motivated by the desire to accommodate existing transformations observed in the literature and admit new ones.
As a start, let $\gG$ (resp., $\GG$) be the set of functions $g$ in $\RR^\FF$ such that $g = W_0 v$ for some $v \in \vV$ (resp., for some $v \in \VV$).
Let $\hH$ (resp., $\HH$) be all functions $h$ in $\RR^\FF$ such that, for
some $v \in \vV$ (resp., for some $v \in \VV$), we have $h(x, a) = H(x, a, v)$ for all $(x, a) \in \FF$. 
Let $M$ be the operator defined at $h \in \hH$ by
\begin{equation}
\label{eq:defm}
(M h) (x) = \sup_{a \in \Gamma(x)} h(x, a).
\end{equation}
The operator $M$ maps elements of $\HH$ into $\VV$. Indeed, 
given $h \in \HH$, there exists by definition a $v \in \VV$ such that $h(x, a) =
H(x, a, v)$ at each $(x, a) \in \FF$. We then have
$Mh = Tv \in \VV$ by \eqref{eq:bop} and Assumption~\ref{a:maxex}. 

A \emph{plan factorization} associated with the dynamic program
described above is a pair of operators $(W_0, W_1)$ such that
\begin{enumerate}
	\item $W_0$ is defined on $\vV$ and takes values in $\gG$, 
	%(hence, $\gG$ denotes the set of functions $g$ in $\RR^\FF$ such that $g = W_0 v$ for some $v \in \vV$)
	\item $W_1$ is defined on $\gG$ and takes values in $\hH$, and
	\item the operator composition $W_1 W_0$ satisfies
	\begin{equation}
	\label{eq:qfac}
	(W_1 \, W_0 \, v)(x, a) = H(x, a, v)
	\quad \text{for all } (x,a) \in \FF, \; v \in \vV.
	\end{equation}
\end{enumerate}

Equation \eqref{eq:qfac} states that $W_0$ and $W_1$ provide a
decomposition (or factorization) of $H$, so that each element $W_i$ implements a 
separate component of the two stage (i.e., present and future)  planning problem 
associated with the Bellman equation.  

Given the definition of $M$ in \eqref{eq:defm} and the factorization
requirement \eqref{eq:qfac}, the Bellman operator $T$ from \eqref{eq:bop} can
now be expressed as
\begin{equation}
\label{eq:dt}
T \, = \, M \; W_1 \; W_0.
\end{equation}
In the following, when we discuss the Bellman operator (and the refactored Bellman operator to be defined below), we confine the domain of the operators $W_0$, $W_1$ and $M$ to $\VV$, $\GG$ and $\HH$, respectively. 
In that circumstance, $T$ is a cycle starting from $\VV$. In particular, $W_0$ maps $\VV$ into $\GG$, and $W_1$ maps $\GG$ into $\HH$ by definition, while $M$ maps $\HH$ into $\VV$ as has been shown above. A visualization is given in Figure~\ref{f:triangle}.

\begin{figure}
	\begin{center}
		\includegraphics[scale=0.75]{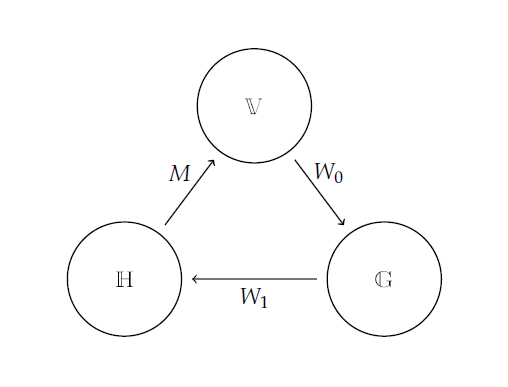}
		\caption{\label{f:triangle} The one-shift operators}
	\end{center}
\end{figure}

%\begin{figure}[h]
%	\centering
%	\begin{minipage}{.5\textwidth}
%		\centering
%		\scalebox{.95}{\input{tikz/triangle.tex}}
%		\caption{The one-shift operators on $\VV$, $\GG$ and $\HH$}
%		\label{f:triangle}
%	\end{minipage}%
%	\begin{minipage}{.5\textwidth}
%		\centering
%		\scalebox{.95}{\input{tikz/triangle2.tex}}
%		\caption{The one-shift operators on $\vV$, $\gG$ and $\hH$}
%		\label{f:triangle2}
%	\end{minipage}
%	%\caption{\label{f:triangle} The one-shift operators}
%\end{figure}

Corresponding to this same plan factorization $(W_0, W_1)$, we introduce a
\emph{refactored Bellman operator} $S$ on $\GG$ defined by
\begin{equation}
\label{eq:dj}
S \, = \,  W_0 \; M \; W_1.
\end{equation}
In terms of Figure~\ref{f:triangle}, $S$ is a cycle starting from $\GG$. 
Corresponding to $S$, we have the \emph{refactored Bellman equation} 
\begin{equation}
\label{eq:rbe}
g(x,a) \,=\, (W_0 \, M \, W_1 \, g) (x,a)
\quad \text{for all } \, (x,a) \in \FF. 
\end{equation}
The transformations of the Bellman equation we wish to consider all correspond
to a version of \eqref{eq:rbe}, under some suitable specification of $W_0$ and
$W_1$.

\subsection{Examples}
\label{ss:ex}

In this section we briefly illustrate the framework with examples.  The
examples contain both common transformations and new ones. 
Throughout, $\beta \in (0,1)$ denotes the discount factor.

\subsubsection{Q-Factors}

\label{sss:ql}

Consider the traditional finite state Markov decision process from Example~\ref{eg:fmdp},
where $\vV$ is all functions from $\XX$ to $\RR$ and $H$ is as given in
\eqref{eq:h}. Let $\VV = \vV$. One special case of the refactored Bellman equation 
is the equation for determining optimal Q-factors that forms
the basis of Q-learning.  If $I$ is the identity map and we set
\begin{equation*}
(W_0 \, v)(x, a) = r(x, a) + \beta \sum_{x' \in \XX} v(x') p(x, a, x')
\quad \text{and} \quad
W_1 = I,
\end{equation*}
then $(W_0, W_1)$ is a plan factorization, since $(W_1 \, W_0 \, v) (x, a)
= H(x, a, v)$ when the latter is given by \eqref{eq:h}.  For this plan
factorization, the refactored Bellman operator $S$ from \eqref{eq:dj} 
becomes $S \,=\, W_0 \, M$, or, more explicitly,
\begin{equation*}
(S g)(x, a) = 
r(x, a) + \beta \sum_{x' \in \XX} \max_{a' \in \Gamma(x')} g(x', a') p(x, a, x').
\end{equation*}
Solving for $g = S g$ is exactly the problem associated with computation of
optimal Q-factors. See, for example, Equation (1.4) of \cite{bertsekas2012q}
or Equation (2.24) of \cite{bertsekas2017dynamic}. Thus, the underlying
functional equations and fixed point problems of Q-learning correspond to the
special case of our theory where $W_1$ is the identity.

\subsubsection{The Standard Case}

\label{sss:sc}

The previous example showed how the problem of obtaining Q-factors is one kind
of plan factorization---in fact an extreme kind, where the element $W_1$ of the
pair $(W_0, W_1)$ is the identity.  The other extreme case is when $W_0$ is
the identity.  Then, the standard and refactored operators $T$ and $S$
coincide, as can be confirmed by consulting \eqref{eq:dt} and
\eqref{eq:dj}.  Thus, the standard Bellman equation and the Q-factor equation
can be viewed as the two extremal cases of plan factorization.

\subsubsection{Expected Value Transformation}

\label{sss:evt}

Continuing with the Markov decision process from Example~\ref{eg:fmdp},
where $H$ is as given in \eqref{eq:h}, consider the plan factorization $(W_0,
W_1)$ defined by
\begin{equation}
\label{eq:evpf}
(W_0 \, v) (x, a) = \sum_{x' \in \XX} v(x') p(x, a, x')
\quad \text{and} \quad
(W_1\,  g) (x, a) = r(x, a) + \beta g(x, a).
\end{equation}
Evidently \eqref{eq:evpf} yields a plan
factorization, in the sense that $(W_1 \, W_0 \, v)(x, a)
= H(x, a, v)$ when the latter is given by \eqref{eq:h}.  For this plan
factorization, the refactored Bellman operator $S$ from \eqref{eq:dj} 
becomes 
\begin{equation}
\label{eq:evpfs}
(S g)(x, a) = 
(W_0 \, M \, W_1 \, g) (x, a) = 
\sum_{x' \in \XX}
\max_{a' \in \Gamma(x')}
\left\{ r(x', a') + \beta   g(x', a') \right\} p(x, a, x').
\end{equation}
With this choice of plan factorization, 
the operator $S$ is a generalization of the ``expected value operator'' used by
\cite{rust1987optimal} in the context of optimal timing of decisions for bus
engine replacement. At this point in time it is not apparent why one should
favor the use of $S$ over the regular Bellman operator $T$, and indeed there
might be no advantage.  The benefit of having a general optimality
theory based around $S$ rather than $T$, as constructed below, is the option
of using $S$ when the structure of the problem implies that doing so \emph{is}
advantageous.  Examples are shown in Section~\ref{s:aii}.

\subsubsection{Optimal Stopping}

\label{sss:os}

Consider an optimal stopping problem with additively separable preferences (see, e.g., \cite{monahan1980optimal} or
\cite{peskir2006optimal}), where an agent observes current state $x$ and
chooses between stopping and continuing. Stopping generates terminal reward
$r(x)$ while continuing yields flow continuation reward $c(x)$. If the agent
continues, the next period state $x'$ is drawn from $P(x, \diff x')$ and the process repeats. 
The agent takes actions in order to maximize expected lifetime rewards.
%The Bellman equation is
%    $v(x) = \max \left\{ r(x), \, c(x) + \beta \int v(x') P(x, \diff x')
%    \right\}$.
The state-action aggregator $H$ for this model is
\begin{equation}
\label{eq:hos}
H(x, a, v) = a \, r(x) + (1 - a) \left[ c(x) + \beta \int v(x') P(x, \diff x') \right],
\end{equation}
where $a \in \Gamma(x) := \{0, 1\}$ is a binary choice variable with $a=1$
indicating the decision to stop. Let $\XX$ be a Borel subset of $\RR^m$, $r$
and $c$ be bounded and continuous, $\vV = \VV$ be the set of bounded continuous
functions on $\XX$, and $P$ satisfy the Feller property (i.e., $x \mapsto \int
v(x') P(x, \diff x')$ is bounded and continuous whenever $v$ is). 
Assumption~\ref{a:maxex} holds because the action space is discrete and 
Assumption~\ref{a:ess} holds by a 
standard contraction argument (see, e.g., Section~2.1 of \cite{bertsekas2013abstract}). 

One possible plan factorization is 
\begin{equation*}
\label{eq:fac_os}
(W_0 \, v) (x) = \int v(x') P(x, \diff x')
\quad \text{and} \quad
(W_1\,  g) (x) = 
a \, r(x) + (1 - a) \left[ c(x) + \beta g(x) \right].
\end{equation*}
The refactored Bellman operator $S \,=\, W_0 \, M \, W_1$ is then given by
\begin{equation}
\label{eq:rbo2_os}
(Sg)(x) = \int \max 
\left\{ r(x'), \, c(x') + \beta g(x') \right\} 
P(x, \diff x').
\end{equation}
There exist applications in the field of optimal stopping where iteration with
the operator $S$ given in \eqref{eq:rbo2_os} is computationally more efficient
than iterating with the standard Bellman operator $T$ (see, e.g., \cite{rust1987optimal}). 
Section~\ref{ss:os} expands on this point.

\subsubsection{Risk Sensitive Preferences}

\label{sss:rsk}

Consider a risk-sensitive control problem  where $H$ can be expressed as
\begin{equation}
\label{eq:rsk}
H(x, a, v) =
\left\{
r(x, a) - \frac{\beta}{\gamma} \log 
\int \exp[ - \gamma v(x') ] P(x, a, \diff x') 
\right\}.
\end{equation}
Here $P(x, a, \diff x')$ gives one-step state transition
probabilities given state $x \in \RR_+$ and action $a \in [0, x]$,
while $\gamma > 0$ controls the degree of risk sensitivity.  One example of this set up occurs in
\cite{bauerle2018stochastic}, where $H(x, a, v)$ has the interpretation of
lifetime rewards when an agent owns current capital $x$, choses current
consumption $a$, receives current reward $r(x, a)$ and then transitions to the
next period with capital drawn from $P(x, a, \diff x')$. Sufficient conditions for 
Assumptions~\ref{a:maxex}--\ref{a:ess} can be found in Theorem~1 of that paper. 
%The function $v$ is the continuation value function.

There are many plan factorizations we could consider here, one of which 
extends the Q-factor transformation in Section~\ref{sss:ql}, and
another of which extends the expected value transformation in
Section~\ref{sss:evt}.  Taking the latter as an example, 
we set 
\begin{equation*}
(W_0 \, v) (x, a) = 
- \frac{\beta}{\gamma} \log
\int \exp[ - \gamma v(x') ] P(x, a, \diff x')
\end{equation*}
\begin{equation}
\label{eq:evrsk}
%	(W_0 \, v) (x, a) = 
%	- \frac{\beta}{\gamma} \log
%	\int \exp[ - \gamma v(x') ] P(x, a, \diff x') 
\text{and} \quad
(W_1\,  g) (x, a) = r(x, a) + g(x, a).
\end{equation}
Evidently \eqref{eq:evrsk} yields a plan
factorization, in the sense that $(W_1 \, W_0 \, v)(x, a)
= H(x, a, v)$ when the latter is given by \eqref{eq:rsk}.
The refactored Bellman operator $S \,=\, W_0 \, M \, W_1$ is given by
\begin{equation}
\label{eq:rfrsk}
(S g)(x, a) = 
- \frac{\beta}{\gamma} \log
\int \exp \left[ 
- \gamma 
\max_{a' \in \Gamma(x')} \{ r(x', a') + g(x', a') \}
\right]
P(x, a, \diff x').
\end{equation}

\subsection{Refactored Policy Values}

Returning to the general setting, we also wish to consider the value of
individual policies under the transformation associated with a given plan
factorization $(W_0, W_1)$.  This will be important when we consider topics
such as policy function iteration.  To this end, for each $\sigma \in \Sigma$ and
each $h \in \hH$, we define the operator $M_\sigma \colon \hH \to \vV$ by
\begin{equation}
\label{eq:defms}
M_\sigma h (x) := h (x, \sigma(x)).
\end{equation}
That $M_\sigma$ does in fact map $\hH$ into $\vV$ follows from the definition
of these spaces and \eqref{eq:f}. In particular, if $h \in \hH$, by
definition, there exists a $v \in \vV$ such that $h(x, a) = H(x, a, v)$ for
all $(x, a) \in \FF$. Then $M_\sigma h \in \vV$ follows directly from
\eqref{eq:f}. Comparing with $M$ defined in \eqref{eq:defm}, we have
\begin{equation}
\label{eq:mvsms}
M_\sigma h \leq M h
\quad \text{for each } h \in \hH \text{ and } 
\sigma \in \Sigma.
%    \quad \text{and} \quad
%    M_\sigma h = M h \,
%    \text{ for at least one } \sigma \in \Sigma.
\end{equation}
%
%The inequality in \eqref{eq:mvsms} is obvious.  To see that the existence
%claim holds, observe that, by the definition of $\HH$, there exists a $v
%\in \VV$ such that $h = W_1 W_0 v$.  For this $v$, by
%Assumption~\ref{a:maxex}, there exists a $\sigma$ in $\Sigma$ such that
%$M_\sigma \, W_1  \,W_0  \,v = M  \,W_1 \, W_0 \, v$.  Since 
%$h = W_1 \, W_0 \, v$, this verifies the second claim in \eqref{eq:mvsms}.
%
Given $\sigma \in \Sigma$, the operator $T_\sigma$ from $\vV$ to itself defined by
$T_\sigma v(x) = H(x, \sigma(x), v)$ for all $x \in \XX$ will be called the
\emph{$\sigma$-value operator}. By construction, it has the property that $v_\sigma$
is the $\sigma$-value function corresponding to $\sigma$ if and only if it is
a fixed point of $T_\sigma$. With the notation introduced above, we can
express it as
\begin{equation}
\label{eq:sto}
T_\sigma \, = \, M_\sigma \; W_1 \; W_0.
\end{equation}
Note that $T_\sigma$ is a cycle starting from $\vV$. In particular, $W_0$ maps $\vV$ into $\gG$, and $W_1$ maps $\gG$ into $\hH$ by the definition of plan factorization, while $M_\sigma$ maps $\hH$ into $\vV$ as shown above. 

Analogous to the definition of the refactored Bellman operator in \eqref{eq:dj},
we introduce the \emph{refactored $\sigma$-value operator} on $\gG$
\begin{equation}
\label{eq:sso}
S_\sigma \, := \, W_0 \; M_\sigma \; W_1
\end{equation}
corresponding to a given plan factorization $(W_0, W_1)$. Similarly, $S_\sigma$ is a cycle starting from $\gG$.
A fixed point $g_\sigma$ of $S_\sigma$ is called a \emph{refactored
	$\sigma$-value function}.  The value $g_\sigma(x, a)$ can be interpreted as
the value of following policy $\sigma$ in all subsequent periods, conditional
on current action $a$.

\begin{example}
	\label{eg:ssig}
	Recall the expected value function factorization of a finite state
	Markov decision process discussed in Section~\ref{sss:evt}.
	For a given policy $\sigma$, the refactored $\sigma$-value operator
	can be obtained by replacing $M$ in \eqref{eq:evpfs} with $M_\sigma$,
	which yields
	\begin{align}
	\label{eq:ssigfmdp}
	(S_\sigma g)(x, a) 
	&= (W_0 \, M_\sigma \, W_1 \, g) (x, a) \nonumber \\
	&= \sum_{x' \in \XX}
	\left\{ r(x', \sigma(x')) + \beta   g(x', \sigma(x')) \right\} p(x, a, x').
	\end{align}
\end{example}

\begin{example}
	\label{eg:osrf}
	Recall the plan factorization \eqref{eq:rbo2_os} applied to the optimal
	stopping problem with aggregator $H$ given by \eqref{eq:hos}.
	A feasible policy is a map $\sigma$ from $\XX$ to $\AA = \{0, 1\}$.
	The refactored $\sigma$-value operator corresponding to this plan
	factorization is 
	\begin{equation}
	\label{eq:osrfo}
	(S_\sigma g)(x) = \int 
	\left\{ \sigma(x') r(x') + (1 - \sigma(x'))[ c(x') + \beta g(x')] \right\} 
	P(x, \diff x').
	\end{equation}
\end{example}

\subsection{Fixed Points and Iteration}

\label{ss:fps}

We begin with some preliminary results on properties of the operators defined
above.
Our first result states that, to iterate with $T$, one can
alternatively iterate with $S$, since iterates of $S$ can be converted
directly into iterates of $T$.  Moreover, the converse is also true, and
similar statements hold for the policy operators:

\begin{proposition}
	\label{p:iter}
	For every $n \in \NN$, we have
	\begin{equation*}
	S^n \, = \, W_0 \; T^{n-1} \; M \; W_1
	\quad \text{and} \quad
	T^n \, = \, M \; W_1 \; S^{n-1} \; W_0.
	\end{equation*}
	Moreover, for every $n \in \NN$ and every $\sigma \in \Sigma$, we have
	\begin{equation*}
	S_\sigma^n \, = \, W_0 \; T_\sigma^{n-1} \; M_{\sigma} \; W_1
	\quad \text{and} \quad
	T_\sigma^n \, = \, M_{\sigma} \; W_1 \; S_\sigma^{n-1} \; W_0.
	\end{equation*}
\end{proposition}

Proposition~\ref{p:iter} follows from the definitions of $T$ and $S$ along with a simple induction argument.  While the proof is relatively straightforward, the result is important because that it provides a
metric-free statement of the fact that the sequence of iterates of any refactored Bellman operator converges at the same rate as that of the standard Bellman operator.

The next result shows a fundamental connection between the two forms of the Bellman operator in terms of their fixed points.

\begin{proposition}
	\label{p:exu}
	The Bellman operator $T$ admits a unique fixed point
	$\bar v$ in $\VV$ if and only if the refactored Bellman operator $S$ admits a unique fixed
	point $\bar g$ in $\GG$.  Whenever these fixed points exist, they are related
	by $\bar v = M \, W_1 \, \bar g$ and $\bar g = W_0 \, \bar v$.
\end{proposition}

Thus, if a unique fixed point of the Bellman operator is desired but $T$ is
difficult to work with, a viable option is to show that $S$ has a unique fixed point,
compute it, and then recover the unique fixed point of $T$ via $\bar v = M \,
W_1 \, \bar g$. 

A result analogous to Proposition~\ref{p:exu} holds for the policy operators:

\begin{proposition}
	\label{p:vsigi}
	Given $\sigma \in \Sigma$, the refactored $\sigma$-value operator
	$S_\sigma$ has a unique fixed point $g_\sigma$ in $\gG$ if and only if
	$T_\sigma$ has a unique fixed point $v_\sigma$ in $\vV$.  Whenever these
	fixed points exist, they are related by $v_\sigma = M_\sigma \, W_1 \,
	g_\sigma$ and $g_\sigma = W_0 \, v_\sigma$.
\end{proposition}

Proposition~\ref{p:vsigi} implies that, under Assumption~\ref{a:ess}, 
there exists exactly one refactored $\sigma$-value function
$g_\sigma$ in $\gG$ for every $\sigma \in \Sigma$. 
Proposition~\ref{p:vsigi} is also useful for the converse
implication.  In particular, to establish Assumption~\ref{a:ess}, which is
often nontrivial when preferences are not additively separable, we can
equivalently show that $S_\sigma$ has a unique fixed point in $\gG$.

\section{Optimality}

\label{s:opt}

Next we turn to optimality, with a particular focus on the properties that
must be placed on a given plan factorization in order for the corresponding
(refactored) Bellman equation to lead us to optimal actions.
Our first step, however, is to define optimality and recall some standard results.

\subsection{Fixed Points and Optimal Policies}

The \textit{value function} associated with our dynamic program is defined at $x \in
\XX$ by
\begin{equation}
\label{eq:szvf}
v^*(x) = \sup_{\sigma \in \Sigma} v_{\sigma} (x).
\end{equation}
A feasible policy $\sigma^*$ is called \textit{optimal} if 
$v_{\sigma^*} = v^*$ on $\XX$.  Given $\sigma \in \Sigma$, \emph{Bellman's principle of optimality}
states that 
\begin{equation}
\label{eq:pop}
\sigma \text{ is an optimal policy } \iff \; \sigma \text{ is $v^*$-greedy}.
\end{equation}

The next theorem is a simple extension of foundational optimality results for
dynamic decision problems that
link the Bellman equation to optimality (see, e.g., \cite{bertsekas2013abstract}). Its proof is omitted.  
The assumptions of Section~\ref{ss:paa} are taken to be in force.

\begin{theorem}
	\label{t:bpost}
	The next two statements are equivalent:
	\begin{enumerate}
		\item The value function $v^*$ lies in $\VV$ and satisfies
		the Bellman equation \eqref{eq:be}.
		\item Bellman's principle of optimality \eqref{eq:pop} holds and the set of optimal
		policies is nonempty.
	\end{enumerate}
\end{theorem}

The motivation behind the transformations we consider in this paper is that
the Bellman equation can be refactored into a more convenient form without
affecting optimality.  Thus, it is natural to ask when---and under what
conditions---a result analogous
to Theorem~\ref{t:bpost} holds for the refactored Bellman equation
\eqref{eq:rbe}.  We will show that an analogous result obtains whenever a form of
monotonicity holds for the transformation being used in the plan
factorization.  

Before we get to this result, however, we need to address
the following complication: there are two
distinct functions that can take the part of $v^*$ in Theorem~\ref{t:bpost}. One
is the ``rotated value function''
\begin{equation}
\label{eq:defghat}
\hat g :=  W_0 \, v^*.
\end{equation}
The second is
\begin{equation}
\label{eq:svf}
g^*(x, a) := \sup_{\sigma \in \Sigma} g_\sigma(x, a).
\end{equation}
We call $g^*$ the \emph{refactored value function}.
The definition of $g^*$ directly parallels the definition of the value function in
\eqref{eq:szvf}, with $g^*(x, a)$ representing the maximal value that can be obtained from state
$x$ conditional on choosing $a$ in the current period. 
(Since Assumption~\ref{a:ess} is in force, the
set of functions $\{g_\sigma\}$ in the definition of $g^*$ is well defined
by Proposition~\ref{p:vsigi}, and hence so is $g^*$ as an extended
real-valued function, although it might or might not live in $\gG$.)

As shown below, the functions $\hat g$ and $g^*$ are not in general equal,
although they become so when a certain form of monotonicity is imposed.
Moreover, under this same monotonicity condition, if one and hence both of
these functions are fixed points of $S$, we obtain valuable optimality
results.  

In stating these results, we recall that a map $A$ from one
partially ordered set $(E, \leq)$ to another such set $(F, \leq)$ is
called \emph{monotone} if $x \leq y$ implies $Ax \leq Ay$.  Below,
monotonicity is with respect to the pointwise partial order on $\VV$,
$\GG$ and $\HH$. Moreover, given $g \in \GG$, a policy $\sigma \in \Sigma$ is
called \emph{$g$-greedy} if $M_\sigma \, W_1 \, g = M \, W_1 \, g$.
At least one $g$-greedy policy exists for every $g \in \GG$. Indeed, $g \in
\GG$ implies the existence of a $v \in \VV$ such that $g = W_0 \, v$, and to
this $v$ there corresponds a $v$-greedy policy $\sigma$ by
Assumption~\ref{a:maxex}. At this $\sigma$ we have $H(x, \sigma(x), v) =
\sup_{a \in \Gamma(x)} H(x, a, v)$ for every $x \in \XX$, or, equivalently,
$M_\sigma \, W_1 \, W_0 \, v = M \, W_1 \, W_0 \, v$ pointwise on $\XX$. Since
$g = W_0 \, v$, this policy is $g$-greedy.

We can now state our first significant optimality result.
It shows that, under some monotonicity conditions, the
standard Bellman equation and the refactored Bellman equation have ``equal
rights,'' in the sense that both can be used to obtain the same optimal
policy, and both satisfy a version of Bellman's principle of optimality.

\begin{theorem}
	\label{t:bpowm}
	Let $(W_0, W_1)$ be a plan factorization, let $\hat g$ be
	as defined in \eqref{eq:defghat} and let $g^*$ be as defined in
	\eqref{eq:svf}.
	If $W_0$ and $W_1$ are both monotone, then the following
	statements are equivalent:
	\begin{enumerate}
		\item $g^*$ lies in $\GG$ and satisfies the refactored Bellman equation \eqref{eq:rbe}.
		\item $v^*$ lies in $\VV$ and satisfies the Bellman equation \eqref{eq:be}.
	\end{enumerate}
	If these conditions hold, then $g^* = \hat g$,
	the set of optimal policies is nonempty and, for $\sigma \in \Sigma$, we
	have
	\begin{equation}
	\label{eq:pa}
	\sigma \text{ is an optimal policy}
	\iff
	\sigma \text{ is $g^*$-greedy}
	\iff
	\sigma \text{ is $v^*$-greedy}.
	\end{equation}
\end{theorem}

\vspace{0.2em}

The monotonicity conditions on $W_0$ and $W_1$ clearly hold in all of the
example transformations discussed in Sections~\ref{sss:ql}--\ref{sss:os}. 
It is possible to envisage settings where they fail, however.
For example, if we replace the factorization $(W_0, W_0)$ in \eqref{eq:evrsk}
with 
\begin{equation*}
(W_0 \, v) (x, a) = 
\int \exp[ - \gamma v(x') ] P(x, a, \diff x') 
\end{equation*}
\begin{equation*}
%	(W_0 \, v) (x, a) = 
%	\int \exp[ - \gamma v(x') ] P(x, a, \diff x') 
\text{and} \quad
(W_1\,  g) (x, a) = r(x, a) - \frac{\beta}{\gamma} \log g(x, a),
\end{equation*}
%
%\begin{equation*}
%(W_0 \, v) (x, a) = 
%\frac{\beta}{\gamma} \log
%\int \exp[ - \gamma v(x') ] P(x, a, \diff x') 
%\quad \text{and} \quad
%(W_1\,  g) (x, a) = r(x, a) - g(x, a),
%\end{equation*}
%%
we again have a valid plan factorization (since $(W_1 \, W_0 \, v)(x, a)
= H(x, a, v)$ when the latter is given by \eqref{eq:rsk}) but neither $W_0$
nor $W_1$ is monotone.

In general, monotonicity of $W_0$ and $W_1$ cannot be dropped from
Theorem~\ref{t:bpowm} without changing its conclusions.  In
Section~\ref{ss:counter} of the Appendix we exhibit dynamic programs and plan
factorizations that illustrate the following possibilities:
\begin{enumerate}
	\item $\hat g \not= g^*$.
	\item $Tv^* = v^*$ and yet $S g^* \not= g^*$.
	\item $S g^* = g^*$ and yet $Tv^* \not= v^*$.
\end{enumerate}

\subsection{Sufficient Conditions}

\label{s:suffcon}

Theorem~\ref{t:bpowm} tells us that if the stated monotonicity condition holds
and $S g^* = g^*$, then we can be assured of the existence of optimal policies
and have a means to characterize them.  What we lack is a set of sufficient
conditions under which the refactored Bellman operator has a unique fixed
point that is equal to $g^*$.  The theorem stated in this section fills that
gap.

To state the theorem, we recall that a self-mapping $A$ on a topological space $U$ is
called \emph{asymptotically stable} on $U$ if $A$ has a unique fixed point
$\bar u$ in $U$ and $A^n u \to \bar u$ as $n \to \infty$ for all $u \in U$.
In the following, we assume that there exists a Hausdorff topology $\tau$ on
$\gG$ under which the pointwise partial order is closed (i.e., its graph is
closed in the product topology on $\gG \times \gG$. The key implication is
that if $f_n \to f$ and $g_n \to g$ under $\tau$ and $f_n \leq g_n$ for all
$n$, then $f \leq g$). Asymptotic stability is with respect to this topology.

In the theorem below, $(W_0, W_1)$ is a given plan factorization and $g^*$ is
the refactored value function.

\begin{theorem}
	\label{t:asc}
	If $W_0$ and $W_1$ are monotone, $S$ is asymptotically stable on $\GG$ and $\{S_\sigma\}_{\sigma \in \Sigma}$ are asymptotically stable on $\gG$, then
	\begin{enumerate}
		\item $g^*$ is the unique solution to the refactored Bellman equation in
		$\GG$,
		\item $\lim_{k \to \infty} S^k g = g^*$ under $\tau$ 
		for all $g \in \GG$,
		\item at least one optimal policy exists, and
		\item a feasible policy is optimal if and only if it is $g^*$-greedy.
	\end{enumerate}
\end{theorem}

In Theorem~\ref{t:asc} we eschewed a contractivity assumption on $S$ or
$S_\sigma$, since such an assumption is
problematic in certain applications (see, e.g., \cite{bertsekas2013abstract}).
Nevertheless, contraction methods can be applied to many other problems.  For
this reason we add the next proposition, which is a refactored analog of the
classical theory of dynamic programming based around contraction mappings.

Let $\| \cdot \|_\kappa$ be defined at $f \in \RR^\FF$ by $\| f \|_\kappa =
\sup \left| f/\kappa \right|$.  Here the supremum is over all $(x, \, a) \in
\FF$ and $\kappa$ is a fixed ``weighting function'', i.e., an element of
$\RR^\FF$ satisfying $\inf \kappa > 0$ on $\FF$.  This mapping defines a norm on
$b_\kappa \FF$, the set of $f \in \RR^\FF$ such that $\| f \|_\kappa <
\infty$.  

\begin{proposition}
	\label{p:cc}
	Let $W_0$ and $W_1$ be monotone. If $\gG$ and $\GG$ are closed subsets of $b_\kappa \FF$ and there exists a positive constant $\alpha$ such that $\alpha < 1$ and
	\begin{equation}
	\label{eq:ucons}
	\| S_\sigma \, g - S_\sigma \, g' \|_\kappa \leq \alpha \,  \| g - g' \|_\kappa
	\;\;
	\text{ for all } \, g, g' \in \gG \text{ and } \sigma \in \Sigma,
	\end{equation}
	then $S$ is a contraction mapping on $(\GG, \| \cdot \|_\kappa)$ of modulus $\alpha$, $S$ is asymptotically stable on $\GG$, and the conclusions of Theorem~\ref{t:asc} all hold.
\end{proposition}

The purpose of $\kappa$ here is to control for unbounded rewards (see, e.g,
\cite{bertsekas2013abstract}).  When rewards are bounded we can take $\kappa
\equiv 1$, in which case $\| \cdot \|_\kappa$ is the ordinary supremum norm
and $b_\kappa \FF$ is just the bounded functions on $\FF$.
Notice also that the contractivity requirement in \eqref{eq:ucons} is imposed
on $S_\sigma$ rather than $S$, making the condition easier to verify (since
$S_\sigma$ does not involve a maximization step).

\subsection{Policy Iteration}

\label{ss:rfpi}

There are algorithms besides value function iteration that achieve faster
convergence in some applications.  These include Howard's policy iteration
algorithm and a popular variant called optimistic policy iteration (see, e.g.,
\cite{puterman1982action} or \cite{bertsekas2017dynamic}). Optimistic policy
iteration is important because it includes value function iteration and
Howard's policy iteration algorithm as limiting cases.  In this section we
give conditions under which optimistic policy iteration is successful in the
setting of a given plan factorization $(W_0, W_1)$. We make the following assumption.

\begin{assumption}
	\label{a:exg}
	At least one $v$-greedy policy exists in $\Sigma$ for each $v$ in $\vV$.
\end{assumption}

Note that Assumption~\ref{a:exg} implies Assumption~\ref{a:maxex}. In particular, by Assumption~\ref{a:exg}, for each $v \in \vV$, there exists a $\sigma$ in $\Sigma$ such that $Tv (x) = H(x, \sigma(x), v)$ for all $x \in \XX$. \eqref{eq:f} then implies that $Tv \in \vV$. Therefore, $T$ is a self-map on $\vV$ and Assumption~\ref{a:maxex} holds by letting $\VV = \vV$.

The standard algorithm starts with an initial candidate $v_0 \in \vV$ and
generates sequences $\{ \sigma_k^v\}$, $\{\Sigma_k^v \}$ in $\Sigma$ and
$\{v_k \}$ in $\vV$ by taking
\begin{equation}
\label{eq:mpi}
\sigma_k^v \in \Sigma_k^v
\quad \text{and} \quad
v_{k+1} = T_{\sigma_k^v}^{m_k} \, v_k
\quad \text{for all } \, 
k \in \NN_0,
\end{equation}
where $\Sigma_k^v$ is the set of $v_k$-greedy policies, and 
$\{ m_k\}$ is a sequence of positive integers. 
The first step of equation \eqref{eq:mpi} is called \textit{policy improvement}, 
while the second step is called \textit{partial policy evaluation}. 
If $m_k = 1$ for all $k$ then the algorithm reduces to value function
iteration. 

To extend this idea to the refactored case, we take an initial candidate $g_0
\in \gG$ and generate sequences $\{\sigma_k^g\}$, $\{\Sigma_k^g \}$ in $\Sigma$
and $\{ g_k \}$ in $\gG$ via
\begin{equation}
\label{eq:mpi2}
\sigma_k^g \in \Sigma_k^g
\quad \text{and} \quad
g_{k+1} = S_{\sigma_k^g}^{m_k} \, g_k
\quad \text{for all } \, 
k \in \NN_0,
\end{equation}
where $\Sigma_k^g$ is the set of $g_k$-greedy policies. The next result shows that \eqref{eq:mpi} and \eqref{eq:mpi2} indeed generate the same sequences of greedy policies.

\begin{theorem}
	\label{t:mpi}
	If $v_0 \in \vV$ and $g_0 = W_0 \, v_0$, then 
	sequences $\{ \sigma_k \}$ and $\{ \Sigma_k \}$ in $\Sigma$ satisfy 
	\eqref{eq:mpi} if and only if they satisfy \eqref{eq:mpi2}.
	Moreover, the generated $\{v_k\}$ and $\{g_k\}$ sequences satisfy $g_k = W_0  \, v_k$ for all $k$.
\end{theorem}

%The next theorem pertains to a sequence generated by
%\eqref{eq:mpi2}.  In the statement of the theorem, asymptotic
%stability is with respect to a topology $\tau$ on $\GG$ under which the
%pointwise partial order is closed.

Moreover, optimistic policy iteration via the refactored Bellman operator converges, as for the standard Bellman operator:

\begin{theorem}
	\label{t:hpi_conv}
	Let $W_0$ and $W_1$ be monotone and let $S$ and $\{S_\sigma \}_{\sigma \in \Sigma}$ be asymptotically stable on $\gG$.  
	If $g_0 \leq S g_0$, then $g_k \leq g_{k+1}$ for all $k$ and $g_k \to g^*$ 
	as $k \to \infty$.
\end{theorem}

\begin{example}
	\label{eg:ssigc}
	In \eqref{eq:ssigfmdp}, we gave the 
	refactored $\sigma$-value operator $S_\sigma$ associated with the 
	expected value function factorization of a finite state
	Markov decision process. Recall that we set $\vV$ and $\VV$ to be the set of functions from $\XX$ to $\RR$. So Assumption~\ref{a:exg} holds based on the analysis of Section~\ref{ss:paa}. It is straightforward to check that 
	this choice of $S_\sigma$ satisfies \eqref{eq:ucons} when $\alpha$ is set
	to the discount factor $\beta$ and $\kappa \equiv 1$.  Since the operators
	$W_0$ and $W_1$ associated with this plan factorization are both monotone
	(see \eqref{eq:evpfs}), the conclusions of Theorems~\ref{t:asc},
	\ref{t:mpi} and \ref{t:hpi_conv} are all valid.
\end{example}

\begin{example}
	\label{eg:osrfc}
	The refactored $\sigma$-value operator $S_\sigma$ corresponding to 
	the plan factorization used in the optimal
	stopping problem from Section~\ref{sss:os} was given in 
	\eqref{eq:osrfo}. Since both $\vV$ and $\VV$ are the set of bounded continuous functions on $\XX$, Assumption~\ref{a:exg} has been verified in Section~\ref{sss:os}. Moreover, it is straightforward to show that 
	$S_\sigma$ satisfies \eqref{eq:ucons} when $\alpha := \beta$
	and $\kappa \equiv 1$.  Hence,  for this plan factorization of the optimal
	stopping problem, the conclusions of
	Theorems~\ref{t:asc}--\ref{t:hpi_conv} hold.
\end{example}

\section{Applications}

\label{s:aii}

In this section we connect the theory presented above with applications. In each case, we use transformations of the Bellman equation that reduce dimensionality and enhance computational efficiency. Throughout, we use $\EE_{y|x}$ to denote the expectation of $y$ conditional on $x$.
Code that replicates all of our numerical results can be found at
\url{https://github.com/jstac/dp_deconstructed_code}.

\subsection{Consumer Bankruptcy}
\label{ss:cb}

\cite{livshits2007consumer} analyze consumer bankruptcy rules by building a
dynamic model of household choice with earnings and expense uncertainty.  In
the model (slightly streamlined), a household's time $t$ income is $z_t \eta_t
$, where $z_t$ and $\eta_t$ are the persistent and transitory components of
productivity.  Households face an expense shock $\kappa_t \geq 0$. While $\{
z_t\}$ is Markov, $\{ \eta_t\}$ and $\{ \kappa_t\}$ are {\sc iid}. All are
mutually independent.  
The Bellman equation takes the form
\begin{equation}
\label{eq:bebi}
v(i, d, z, \eta, \kappa) = \max_{c, \, d', \, i'}
\left[
u(c) + \beta \, \EE_{z', \eta',\kappa' \mid z} \, v(i', d', z', \eta', \kappa')
\right],
\end{equation}
where $d$ is debt, $c$ is current consumption, $\beta$ is a discount factor, 
and $u$ is one-period utility. 
%and $\EE_z$ is expectation conditional on $z_t = z$.  
Primes denote next period values.  The state
$i$ indicates repayment status and lies in $\{R, B, E\}$, where $R$, $B$ and $E$
correspond to repayment, bankruptcy and default on current expenses.
With $q$ and $\bar r$ indicating the interest rate for a household in states
$R$ and $E$ respectively,
and $\gamma$ parameterizing compulsory
repayments out of current income, the constraints can be expressed as follows:
All variables are nonnegative and, in addition,
\begin{itemize}
	\item[C1.]  if $i=R$, then $c + d + \kappa =  z \eta + q (d', z) d'$ and
	$i' \in \{R, B\}$.
	\item[C2.]  If $i=B$, then $c = (1 - \gamma)  z \eta$, $d' = 0$ and $i' \in
	\{R, E\}$.
	\item[C3.]  If $i=E$, then $c = (1 - \gamma)  z \eta$, $d' = (\kappa -
	\gamma  z \eta) (1 + \bar{r})$ and $i' \in \{R, B\}$.
\end{itemize}
A full interpretation of these constraints and the decision problem of the
household can be found on p.~407 of \cite{livshits2007consumer}.
As in that paper, to implement the model computationally, we assume that all
states are discrete.  In particular, $\eta$, $\kappa$ and $z$ take values in
finite sets, and the choice of debt level $d$ is restricted to a finite grid
with maximum $\bar d$.  The model then becomes a special case of the finite
state Markov decision process described in Example~\ref{eg:fmdp}.  The
one-period reward is $u(c)$ and the feasible correspondence $\Gamma$ is
defined by C1--C3 above.  The state is $x = (i, d, z, \eta, \kappa)$ and the
state space $\XX$ is the Cartesian product of $\{R, B, E\}$, for $i$, and the
grids for the remaining variables.   The action is $a = (i', d', c)$ and the
sets of candidate value functions are $\vV = \VV = \RR^\XX$.

A fast algorithm for solving this model is important because it allows for
effective exploration of the parameter space during the estimation step (by
solving the model many times at different parameterizations).  We
can greatly accelerate value function iteration for this model by implementing
a suitable plan factorization and then iterating with the refactored Bellman
operator $S$ instead of the original Bellman operator $T$.  
To see why this can deliver acceleration, recall from Figure~\ref{f:triangle} and
the definition of $S$ in \eqref{eq:dj} that one iteration of $S$ is a cycle
starting from $\GG = W_0 \VV$, while $T$ is a cycle starting from $\VV$.  If
functions in $\GG$ are simpler than functions in $\VV$ then the former
will, in general, be faster than the latter.
Similar comments apply when we consider applying $S_\sigma$ instead of
$T_\sigma$, as occurs during policy function iteration.

For example, suppose that we adopt the plan factorization associated with the
expected value function transformation from Section~\ref{sss:evt}.  In view of
the discussion in Example~\ref{eg:ssigc}, we know that refactored value
function iteration and refactored policy function iteration are convergent
under this plan factorization.
Moreover,  since,
$\GG = W_0 \VV$, a typical element of $\GG$ is a function $g = W_0 \, v$,
which, in view of the definition of the expected value function in
\eqref{eq:evpfs}, has the form $g(x, a) = \sum_{x' \in \XX} v(x') p(x, a,
x')$.  In the setting of \eqref{eq:bebi}, this becomes 
$g(i', d', z) = \EE_{z', \eta', \kappa' \mid z} \, v(i', d', z', \eta', \kappa')$.  Significantly, while
the standard value function has five arguments, this refactored value function
has only three. 

To test the expected speed gains, we perform two groups of numerical
experiments.  In each case we compare traditional value function iteration 
(VFI, iterating with $T$) with refactored value function iteration (RVFI, iterating with
$S$).  Regarding the primitives, we set $\gamma = 0.355$, $\bar r = 0.2$ and $u(c) = c^{1 - \sigma} /
(1 - \sigma)$ with $\sigma = 2.0$. We assume that $\{\eta_t \} \iidsim N(0,
0.043)$ and $\{z_t\}$ follows $\log z_{t+1} = \rho \log z_t + \epsilon_{t+1}$
with $\{\epsilon_t\} \iidsim N(0, \delta^2)$. Both are discretized
via the method of \cite{tauchen1986finite}. We discretize $\{\kappa_t\}$ and
$\{d_t\}$ in equidistant steps, where $\{\kappa_t\}$ is uniformly distributed on
$[0,2]$ and the grid points for $d$ lie in $[0, 10]$. In Group-1, we set
$\rho=0.99$ and $\delta^2 = 0.007$ and compare the time taken of VFI and RVFI
for different levels of grid points and $\beta$ values. In Group-2, we set
$\beta=0.98$ and continue the analysis by comparing RVFI and VFI across
different grid sizes and $(\rho, \delta)$ values. The parameter values are
broadly in line with \cite{livshits2007consumer}. Each scenario, we
terminate iteration when distance between successive iterates falls below $10^{-4}$ under the supremum norm.  
(The code uses Python with Numba on a workstation with a 2.9 GHz
Intel Core i7 and 32GB RAM.)

\begin{table}
	\caption{Time taken in seconds: Group-1 experiments}
	\vspace{-0.85cm}
	\label{tb:cmpl_betavar}
	\noindent 
	\begin{center}
		\begin{tabular}{|c|c|c|c|c|c|c|}
			\hline 
			grid size for $(d,z,\eta,\kappa)$ & method & $\beta=0.94$ & $\beta=0.95$ & $\beta=0.96$ & $\beta=0.97$ & $\beta=0.98$\tabularnewline
			\hline 
			\hline 
			\multirow{3}{*}{$(10,10,10,10)$} & VFI & 20.75 & 24.17 & 30.13 & 40.07 & 59.77\tabularnewline
			\cline{2-7} 
			& RVFI & 0.88 & 0.89 & 1.07 & 1.23 & 1.49\tabularnewline
			\cline{2-7} 
			& Ratio & \textbf{23.58} & \textbf{27.16} & \textbf{28.16} & \textbf{32.58} & \textbf{40.11}\tabularnewline
			\hline 
			\multirow{3}{*}{$(12,12,12,12)$} & VFI & 92.78 & 110.06 & 138.15 & 184.59 & 277.13\tabularnewline
			\cline{2-7} 
			& RVFI & 1.46 & 1.56 & 1.81 & 2.23 & 3.00\tabularnewline
			\cline{2-7} 
			& Ratio & \textbf{63.55} & \textbf{70.55} & \textbf{76.33} & \textbf{82.78} & \textbf{92.38}\tabularnewline
			\hline 
			\multirow{3}{*}{$(14,14,14,14)$} & VFI & 321.53 & 387.00 & 484.94 & 648.58 & 978.66\tabularnewline
			\cline{2-7} 
			& RVFI & 2.55 & 2.91 & 3.49 & 4.65 & 6.40\tabularnewline
			\cline{2-7} 
			& Ratio & \textbf{126.09} & \textbf{132.99} & \textbf{138.95} & \textbf{139.48} & \textbf{152.92}\tabularnewline
			\hline 
			\multirow{3}{*}{$(16,16,16,16)$} & VFI & 1445.53 & 1738.22 & 2175.40 & 2904.84 & 4381.13\tabularnewline
			\cline{2-7} 
			& RVFI & 4.92 & 5.75 & 7.12 & 9.45 & 13.65\tabularnewline
			\cline{2-7} 
			& Ratio & \textbf{293.81} & \textbf{302.30} & \textbf{305.53} & \textbf{307.39} & \textbf{320.96}\tabularnewline
			\hline 
			\multirow{3}{*}{$(18,18,18,18)$} & VFI & 2412.84 & 2889.44 & 3614.84 & 4865.62 & 7266.21\tabularnewline
			\cline{2-7} 
			& RVFI & 10.94 & 12.92 & 16.14 & 21.99 & 33.06\tabularnewline
			\cline{2-7} 
			& Ratio & \textbf{220.55} & \textbf{223.64} & \textbf{223.97} & \textbf{221.27} & \textbf{219.79}\tabularnewline
			\hline 
			\multirow{3}{*}{$(20,20,20,20)$} & VFI & 5591.37 & 6737.69 & 8420.48 & 11355.90 & 17020.04\tabularnewline
			\cline{2-7} 
			& RVFI & 19.78 & 23.80 & 31.16 & 41.00 & 62.09\tabularnewline
			\cline{2-7} 
			& Ratio & \textbf{282.68} & \textbf{283.10} & \textbf{270.23} & \textbf{276.97} & \textbf{274.12}\tabularnewline
			\hline 
		\end{tabular}
		\par\end{center}
\end{table}

\begin{table}
	\caption{Time taken in seconds: Group-2 experiments}
	\vspace{-0.85cm}
	\label{tb:cmpl_rhodel}
	\noindent 
	\begin{center}
		\begin{tabular}{|c|c|c|c|c|c|c|}
			\hline 
			grid size for $(d,z,\eta,\kappa)$ & method & $\rho=0.96$ & $\rho=0.97$ & $\rho=0.98$ & $\rho=0.995$ & \tabularnewline
			\hline 
			\multirow{6}{*}{$(10,10,10,10)$} & VFI & 60.83 & 59.66 & 57.69 & 66.70 & \multirow{3}{*}{$\delta^{2}=0.01$}\tabularnewline
			\cline{2-6} 
			& RVFI & 1.49 & 1.43 & 1.43 & 1.49 & \tabularnewline
			\cline{2-6} 
			& Ratio & \textbf{40.83} & \textbf{41.72} & \textbf{40.34} & \textbf{44.77} & \tabularnewline
			\cline{2-7} 
			& VFI & 68.43 & 62.49 & 59.73 & 75.71 & \multirow{3}{*}{$\delta^{2}=0.04$}\tabularnewline
			\cline{2-6} 
			& RVFI & 1.65 & 1.48 & 1.44 & 1.39 & \tabularnewline
			\cline{2-6} 
			& Ratio & \textbf{41.47} & \textbf{42.22} & \textbf{41.48} & \textbf{54.47} & \tabularnewline
			\hline 
			\multirow{6}{*}{$(12,12,12,12)$} & VFI & 255.11 & 274.48 & 268.12 & 327.55 & \multirow{3}{*}{$\delta^{2}=0.01$}\tabularnewline
			\cline{2-6} 
			& RVFI & 2.98 & 2.89 & 2.96 & 3.36 & \tabularnewline
			\cline{2-6} 
			& Ratio & \textbf{85.61} & \textbf{94.98} & \textbf{90.58} & \textbf{97.49} & \tabularnewline
			\cline{2-7} 
			& VFI & 282.59 & 272.91 & 266.27 & 293.99 & \multirow{3}{*}{$\delta^{2}=0.04$}\tabularnewline
			\cline{2-6} 
			& RVFI & 3.46 & 2.96 & 2.92 & 3.05 & \tabularnewline
			\cline{2-6} 
			& Ratio & \textbf{81.67} & \textbf{92.20} & \textbf{91.19} & \textbf{96.39} & \tabularnewline
			\hline 
			\multirow{6}{*}{$(14,14,14,14)$} & VFI & 881.84 & 993.94 & 934.95 & 1012.90 & \multirow{3}{*}{$\delta^{2}=0.01$}\tabularnewline
			\cline{2-6} 
			& RVFI & 6.25 & 6.53 & 6.18 & 6.66 & \tabularnewline
			\cline{2-6} 
			& Ratio & \textbf{141.09} & \textbf{152.21} & \textbf{151.29} & \textbf{152.09} & \tabularnewline
			\cline{2-7} 
			& VFI & 1045.88 & 892.71 & 928.10 & 984.25 & \multirow{3}{*}{$\delta^{2}=0.04$}\tabularnewline
			\cline{2-6} 
			& RVFI & 7.47 & 6.24 & 6.18 & 6.64 & \tabularnewline
			\cline{2-6} 
			& Ratio & \textbf{140.01} & \textbf{143.06} & \textbf{150.18} & \textbf{148.23} & \tabularnewline
			\hline 
			\multirow{6}{*}{$(16,16,16,16)$} & VFI & 2739.10 & 2722.26 & 3113.68 & 7827.82 & \multirow{3}{*}{$\delta^{2}=0.01$}\tabularnewline
			\cline{2-6} 
			& RVFI & 12.81 & 12.75 & 14.50 & 15.03 & \tabularnewline
			\cline{2-6} 
			& Ratio & \textbf{213.83} & \textbf{213.51} & \textbf{214.74} & \textbf{520.81} & \tabularnewline
			\cline{2-7} 
			& VFI & 2843.06 & 2683.64 & 2984.70 & 7190.32 & \multirow{3}{*}{$\delta^{2}=0.04$}\tabularnewline
			\cline{2-6} 
			& RVFI & 13.19 & 12.60 & 13.63 & 13.69 & \tabularnewline
			\cline{2-6} 
			& Ratio & \textbf{215.55} & \textbf{212.99} & \textbf{218.98} & \textbf{525.22} & \tabularnewline
			\hline 
		\end{tabular}
		\par\end{center}
\end{table}

The results are shown in
Tables~\ref{tb:cmpl_betavar}--\ref{tb:cmpl_rhodel}. 
In particular, \textit{ratio} therein represents the ratio of time taken by VFI to that by RVFI. 
In both groups of experiments, reduced dimensionality 
significantly enhances computational efficiency. The speed improvement of RVFI
over VFI reaches two orders of magnitude under a relatively sparse grid
and becomes larger as the number of grid points increase---which is
precisely when computational efficiency is required.
For example, when $(\rho, \delta) = (0.995, 0.2)$ and there are $16$ grid
points for each state variable, RVFI is 525 times
faster than VFI.

\subsection{Optimal Stopping}
\label{ss:os}

Recall the plan factorization used in the optimal stopping problem from Section~\ref{sss:os}, 
corresponding to \eqref{eq:rbo2_os}. We showed in Example~\ref{eg:osrfc} that
the conclusions of Theorems~\ref{t:asc}--\ref{t:hpi_conv} hold, so we can
freely use either the traditional Bellman operator or the refactored version to
compute the optimal policy.  The best choice depends on relative numerical efficiency
of the two operators.

One setting where the refactored Bellman operator is always more numerically efficient is
when the state can be decomposed as
\begin{equation}
\label{eq:cin}
x_t = (y_t, z_t) \in \YY \times \ZZ \subset \RR^\ell \times \RR^k,
%\quad \text{where } \;
%\text{$(y_{t+1}, z_{t+1})$ and $y_t$ are independent given $z_t$}.
\end{equation}
where $(y_{t+1}, z_{t+1})$ and $y_t$ are independent given $z_t$.

The refactored Bellman operator from \eqref{eq:rbo2_os} then reduces to 
\begin{equation*}
(Sg)(z) = 
\int \max \left\{ r(y',z'), \, c(y',z') + \beta g(z') \right\} 
P(z, \diff (y', z')),
\end{equation*}
while the standard Bellman operator is
\begin{equation*}
(Tv)(y, z) = 
\max \left\{ r(y,z), \, c(y,z) + \beta \int v(y', z') P(z, \diff (y', z')) \right\} . 
\end{equation*}
The key difference is that the functions $v \in \VV$, on which $T$ acts,
depend on both $z$ and $y$, whereas the functions $g \in \GG$ on which $S$
acts depend only on $z$.  In other words, the refactored Bellman operator acts on 
functions defined over a lower dimension space than the original state space $\XX = \YY \times \ZZ$.
Hence, when the conditional independence assumption in \eqref{eq:cin} is
valid, the curse of dimensionality is mitigated by working with $S$ rather
than $T$.  Below we quantify the difference, after giving some examples of
settings where \eqref{eq:cin} holds.

\begin{example}
	\label{eg:real_option}
	Consider the problem of pricing a real or financial option (see, e.g.,
	\cite{dixit1994investment} or \cite{kellogg2014effect}). Let $p_t$ be the
	current price of the asset and let $\{s_t\}$ be a Markov
	process affecting asset price via $p_t = f(s_t, \epsilon_t)$, where 
	$\{\epsilon_t \}$ is an {\sc iid} innovation
	process independent of $\{s_t\}$. 
	Each period, the agent decides whether to exercise the option now (i.e.,
	purchase the asset at the strike price) or wait and reconsider next period. 
	Note that $(p_{t+1}, s_{t+1})$ and $p_t$ are independent given $s_t$.
	Hence, this problem fits into the framework of \eqref{eq:cin} if we
	take $y_t = p_t$ and $z_t = s_t$. 
\end{example}

\begin{example}
	Consider a model of job search (see, e.g.,
	\cite{mccall1970economics}), in which a worker receives current wage offer
	$w_t$ and chooses to either accept and work permanently at that wage, or
	reject the offer, receive unemployment compensation $\eta_t$ and reconsider
	next period. The worker aims to find an optimal decision rule that yields
	maximum lifetime rewards. A typical formulation for the wage and
	compensation processes is that both are functions of an 
	exogenous Markov process $\{s_t\}$, with $w_t =
	f(s_t, \epsilon_t)$ and $\eta_t = \ell (s_t, \xi_t)$ where $f$ and $\ell$ are
	nonnegative continuous functions and $\{\epsilon_t\}$ and $\{\xi_t\}$ are
	independent, {\sc iid} innovation processes (see, e.g., \cite{low2010wage}, \cite{bagger2014tenure}).  
	The state space for the job search
	problem is typically set to $(w_t, \eta_t, s_t)$, where $s_t$ is included
	because it can be used to forecast future draws of $w_t$ and $\eta_t$.
	The conditional independence assumption in \eqref{eq:cin} holds if we take
	$y_t = (w_t, \eta_t)$ and $z_t = s_t$.  
\end{example}

Let us now compare the time complexity of VFI and RVFI, based on iterating $T$ and $S$
respectively.

\subsubsection{Finite State Case} 
\label{ss:DSS}

Let $\YY = \times^\ell_{i=1} \YY^i$ and $\ZZ = \times^k_{j=1} \ZZ^j$, where
$\YY^i$ and $\ZZ^j$ are subsets of $\RR$. Each $\YY^i$ (resp., $\ZZ^j$) is
represented by a grid of $L_i$ (resp., $K_j$) points. Integration operations
in both VFI and RVFI are replaced by summations. We use $\hat{P}$ and $\hat{F}$ 
to denote the probability transition matrices for VFI and RVFI respectively.

Let $L := \Pi_{i=1}^{\ell} L_i$ and $K:= \Pi_{j=1}^{k} K_j$ with $L = 1$ for
$\ell =0$. Let $k>0$. There are $LK$ grid points on $\XX = \YY \times \ZZ$ and
$K$ grid points on $\ZZ$.  The matrix $\hat{P}$ is $(LK) \times (LK)$ and
$\hat{F}$ is $K \times (LK)$. VFI and RVFI are implemented by the operators
$\hat{T}$ and $\hat{S}$ defined respectively by
\begin{equation*}
\hat{T} \vec{v} := \vec{r} \vee (\vec{c} + \beta \hat{P} \vec{v})
\quad \text{and} \quad
\hat{S} \vec{g}_z :=  \hat F \left[ \vec r \vee \left(\vec c + \beta \vec g \right) \right].
%	\hat{Q} \vec{\psi}_y := 
%	\vec{c}_y + 
%	\beta \hat{F} (\vec{r} \vee \vec{\psi}).
\end{equation*}
Here $\vec{q}$ represents a column vector with $i$-th element 
equal to $q(y_i, z_i)$, where $(y_i, z_i)$ is the $i$-th element of the list 
of grid points on $\YY \times \ZZ$. $\vec{q}_z$ denotes 
the column vector with the $j$-th element equal to $q(z_j)$, where $z_j$ is 
the $j$-th element of the list of grid points on $\ZZ$. The 
vectors $\vec{v}$, $\vec{r}$, $\vec{c}$ and $\vec{g}$ are
$(LK) \times 1$, while $\vec{g}_z$ is $K \times 1$.
Moreover, $\vee$ denotes taking maximum.

\subsubsection{Continuous State Case} 
\label{ss:mc}

We use the same number of grid points as before, but now for
continuous state function approximation rather than discretization. 
In particular, we replace the discrete state summation with Monte Carlo
integration. Assume that the transition law of the state process follows
\begin{equation*}
y_{t+1} = f_1 (z_t, \epsilon_{t+1}), 
\; \quad
z_{t+1} = f_2 (z_t, \epsilon_{t+1}),
\; \quad
\{ \epsilon_t \} \iidsim \Phi.
\end{equation*}
After drawing $N$ Monte Carlo samples $U_1, \cdots, U_{N} \iidsim \Phi$, RVFI and VFI are implemented by 
\begin{equation*}
\hat{S} g(z) := \frac{1}{N} \sum_{i=1}^{N} \max \left\{ 
r \left( f_1(z, U_i), f_2(z, U_i) \right), \,
c \left( f_1(z, U_i), f_2(z, U_i) \right) + \beta \phi \langle g \rangle \left(f_2(z, U_i)\right)
\right\}
\end{equation*}
\begin{equation*}
\text{and} \quad
\hat{T} v(y, z) := \max \left\{ 
r(y, z), \,
c(y, z) + \beta \frac{1}{N} \sum_{i=1}^{N} \psi \langle v \rangle
\left( f_1(z, U_i), f_2 (z, U_i) \right)
\right\}.
\end{equation*}
Here $g = \{ g(z) \}$ with $z$ in the set of grid points on 
$\ZZ$, and $v = \{ v(y,z)\}$ with $(y,z)$ in the set of grid points on 
$\YY \times \ZZ$. Moreover, $\phi \langle \cdot \rangle$ and 
$\psi \langle \cdot \rangle$ are interpolating functions 
for RVFI and VFI respectively. For example, $\phi \langle g \rangle (z)$
interpolates the vector $g$ to obtain a function $\phi \langle g \rangle$ 
and then evaluates it at $z$.

\subsubsection{Time Complexity}
\label{ss:time_cmplx}

Table \ref{tb:tm_cmplx} provides the time complexity of RVFI and VFI, estimated by counting the number of floating point operations. Each such operation is assumed to have complexity $\mathcal{O}(1)$,
%\footnote{Floating point 
%	operations are any elementary actions 
%	(e.g., $+$, $\times$, $\vee$, $\wedge$) on or assignments with floating point numbers.
%	If $f$ and $g$ are scalar functions on 
%	$\RR^n$, we write $f(x)= \mathcal{O}(g(x))$ whenever there exist 
%	$C, M>0$ such that $\| x \| \geq M$ implies $|f(x)| \leq C |g(x)|$, 
%	where $\|\cdot \|$ is the sup norm.}
so is function evaluation associated with the model primitives. Moreover, for continuous state case, binary search is used when we evaluate the interpolating function at a given point, and our results hold for linear, quadratic, cubic, and $k$-nearest neighbors interpolations.

\begin{table}
	\caption{Time complexity: VFI v.s RVFI}
	\vspace{-0.35cm}
	\label{tb:tm_cmplx}
	\noindent 
	\begin{center}
		{\small
			\begin{tabular}{|c|c|c|c|c|}
				\hline 
				Complexity & VFI: $1$-loop & RVFI: $1$-loop & VFI: $n$-loop & RVFI: $n$-loop\tabularnewline
				\hline 
				\hline 
				FS & $\mathcal{O}(L^{2}K^{2})$ & $\mathcal{O}(L K^{2})$ & $\mathcal{O}(nL^{2}K^{2})$ & $\mathcal{O}(nL K^{2})$	
				\tabularnewline
				\hline 
				CS & $\mathcal{O}(N L K \log(LK))$ & $\mathcal{O}(N K \log(K))$ & $\mathcal{O}(n N L K \log(LK))$ & $\mathcal{O}(n N K \log(K))$	
				\tabularnewline
				\hline
			\end{tabular} 
		}
		\par\end{center}
\end{table}
For both finite state (FS) and continuous state (CS) cases, 
RVFI provides better performance than VFI. For FS, RVFI is
more efficient than VFI by order $\mathcal{O}(L)$, while for CS, RVFI is so 
by order $\mathcal{O} \left(L \log(LK) / \log(K) \right)$.
For example, if we have $100$ grid points in each dimension, in the FS case,
evaluating a given number of loops will take around $100^{\ell}$ times longer
via VFI than via RVFI, after adjusting for order approximations.
See the Appendix (Section~\ref{ss:pf_cmpl}) for proof of Table \ref{tb:tm_cmplx} results.

\subsection{Robust Control}

\label{ss:robust_control}

In the robust control problem studied by \cite{bidder2012robust} 
(see \cite{hansen2008robustness} for systematic discussion),
an agent is endowed with a benchmark model (of the system transition probabilities) but fears that it is misspecified.
While making decisions, she considers alternative distributions that are
distorted versions of the distribution implied by the benchmark, and balances
the cost that an implicit misspecification would cause against the
plausibility of misspecification. The state-action aggregator is given by
\begin{align*}
H(s, \epsilon, u, v) = 
r(s, \epsilon, u) - \beta \theta \log 
\left[ \EE_{s',\epsilon'|s} \exp 
\left(- \frac{v(s', \epsilon')}{\theta} \right)
\right]    
\quad \text{with} \quad 
s' = f(s, u, \epsilon'),
\end{align*}
where $s_t$ is an endogenous state component and $u_t$ is a vector of
controls.  The full state is $x_t = (s_t, \epsilon_t)$, where $\{
\epsilon_t\}$ is an {\sc iid} innovation process.  Rewards at time $t$ are
$r(s_t, \epsilon_t, u_t)$.  Let $\theta>0$ be the penalty parameter that
controls the degree of robustness. The Bellman operator is
\begin{equation*}
\label{eq:rbstctrol_be0}
T v (s, \epsilon) = \max_{u}
\left\{ 
r(s, \epsilon, u) - 
\beta \theta \log 
\left[ 
\EE_{s',\epsilon'|s} \exp
\left(
- \frac{v(s', \epsilon')}{\theta}
\right)
\right]
\right\}.
\end{equation*}
%
%where $s' = f(s, u, \epsilon')$. 
Consider the plan factorization $(W_0, W_1)$ defined by
\begin{equation*}
W_0 v(s) := - \theta \log 
\left\{
\EE_{s', \epsilon' \mid s} 
\exp 
\left[ 
- \frac{v(s', \epsilon')}{\theta}
\right]
\right\}
\quad \text{and} \quad
W_1 g(s, \epsilon, u) := r(s, \epsilon, u) + \beta g(s).
\end{equation*}
%
%This is a valid plan factorization as required by \eqref{eq:qfac}. 
The refactored Bellman operator $S = W_0 \, M \, W_1$ is then
\begin{equation*}
\label{eq:rbstctrol_be1}
S g(s) = - \theta \log 
\left\{
\EE_{s', \epsilon'|s} \exp     
\left[
- \frac{ \max_{u'} 
	\left\{ 
	r(s', \epsilon', u') + 
	\beta g(s')
	\right\} }{\theta}
\right]
\right\}.
\end{equation*}
The maps $W_0$ and $W_1$ are monotone in the standard pointwise ordering for
real valued functions, and the other assumptions of
Sections~\ref{s:gf}--\ref{s:opt} hold in the setting of \cite{bidder2012robust}, 
so the conclusions of Theorems~\ref{t:asc}--\ref{t:hpi_conv} hold. 
The details are omitted.
While the value function acts on both $s$ and $\epsilon$, the refactored value
function acts only on $s$. Hence, the refactored Bellman operator mitigates
the curse of dimensionality, similar to the outcomes in Sections~\ref{ss:cb}
and \ref{ss:os}.

\section{Conclusion}

This paper presents a systematic treatment of the technique of
transforming Bellman equations associated with discrete time dynamic
programs in order to convert them into more advantageous forms.  We formalized these
transformations in terms of what we called plan factorizations and showed how
this concept encompasses and extends existing transformations from the literature.  We
provided conditions related to monotonicity under which the transformations
preserve the most valuable property of Bellman equations: their link to
optimal value functions and optimal decisions.

The applications presented here focused on how transforming the Bellman equation can
mitigate the curse of dimensionality and thereby boost computational
efficiency.  There are, however, other reasons to consider transformations of
the Bellman equation.  For example, it might be that a given transformation
leads to smoother value functions. Smoother functions are easier to approximate in high
dimensional spaces.  It also appears that some transformations of a given
Bellman equation can shift the problem from a setting of unbounded value
functions to bounded ones, where characterization of optimality becomes
easier.  While these ideas go beyond the scope of the current study, the
theory of plan factorizations presented here should serve as a solid
foundation for new work along these lines.

\section{Appendix}

The appendix provides all remaining proofs.

\subsection{Preliminaries}
\label{ss:aprel}

Let $E_i$ be a nonempty set and let $\tau_i$ be a mapping from
$E_i$ to $E_{i+1}$ for $i = 0, 1, 2$ with addition modulo 3 (a convention we adopt throughout this section). Consider the self-mappings 
\begin{equation*}
F_0 \, := \, \tau_2 \; \tau_1 \; \tau_0,
\quad \;
F_1 \, := \, \tau_0 \; \tau_2 \; \tau_1
\quad \; \text{and} \; \quad
F_2 \, := \, \tau_1 \; \tau_0 \; \tau_2 
\end{equation*}
on $E_0$, $E_1$ and $E_2$ respectively.  We then have
\begin{equation}
\label{eq:oir}
F_{i+1} \, \tau_i \, = \, \tau_i \, F_i \quad \text{on $E_i$ for $i = 0, 1, 2$}.
\end{equation}

\begin{lemma}
	\label{l:mof}
	For each $i = 0, 1, 2$, 
	\begin{enumerate}
		\item if $\varrho$ is a fixed point of $F_i$ in $E_i$, then $\tau_i \, \varrho$ is a
		fixed point of $F_{i+1}$ in $E_{i+1}$.
		\item $F_{i+1}^n \, \tau_i \,=\, \tau_i \, F_i^n$ on $E_i$ for all $n \in \NN$.
	\end{enumerate}
\end{lemma}

\begin{proof}
	Regarding part (a), if $\varrho$ is a fixed point of $F_i$ in $E_i$, then
	\eqref{eq:oir} yields $F_{i+1} \tau_i \, \varrho
	=  \tau_i \, F_{i} \varrho = \tau_i \, \varrho$, so $\tau_i \, \varrho$ is a fixed point of $F_{i+1}$.
	Regarding part (b), fix $i$ in $\{0, 1, 2\}$.
	By \eqref{eq:oir}, the statement in (b) is true at $n=1$. Let it also be true at $n-1$. Then, using \eqref{eq:oir} again, 
	$ F_{i+1}^n \, \tau_i 
	\,=\, F_{i+1}^{n-1} \, F_{i+1} \, \tau_i 
	\,=\, F_{i+1}^{n-1} \, \tau_i \, F_i
	\,=\, \tau_i \, F_{i}^{n-1} \, F_i
	\,=\, \tau_i \, F_{i}^n$.
	%    %
	%    \begin{equation*}
	%        F_{i+1}^n \, \tau_i 
	%        \,=\, F_{i+1}^{n-1} \, F_{i+1} \, \tau_i 
	%        \,=\, F_{i+1}^{n-1} \, \tau_i \, F_i
	%        \,=\, \tau_i \, F_{i}^{n-1} \, F_i
	%        \,=\, \tau_i \, F_{i}^n.
	%    \end{equation*}
	%    %
	We conclude that (b) holds at every $n \in \NN$. 
\end{proof}

\begin{lemma}
	\label{l:exf}
	If $F_i$ has a unique fixed point $\varrho_i$ in $E_i$ for some $i$ in $\{0,
	1, 2\}$, then $\tau_i \, \varrho_i$ is the unique fixed point of $F_{i+1}$ in $E_{i+1}$.
\end{lemma}

\begin{proof}
	We have already proved all but uniqueness.  To see that uniqueness holds,
	fix $i \in \{0, 1, 2 \}$ and
	suppose that $F_i$ has only one fixed point in $E_i$, whereas $F_{i+1}$
	has two fixed points in $E_{i+1}$.
	Denote the fixed points of $F_{i+1}$ by $\varrho$ and $f$.  Applying part (a) of
	Lemma~\ref{l:mof} twice, we see that $\tau_{i+2} \, \tau_{i+1} \, \varrho$ and $\tau_{i+2} \, \tau_{i+1} \, f$
	are both fixed points of $F_{i+3} = F_i$.  Since $F_i$ has only one fixed point, 
	$\tau_{i+2} \, \tau_{i+1} \, \varrho = \tau_{i+2} \, \tau_{i+1} \, f$.
	%    %
	%    \begin{equation*}
	%        \tau_{i+2} \, \tau_{i+1} \, e = \tau_{i+2} \, \tau_{i+1} \, f.
	%    \end{equation*}
	%    %
	Applying $\tau_i$ to both sides of the last equality gives $F_{i+1} \varrho = F_{i+1}
	f$.  Since $\varrho$ and $f$ are both fixed points of $F_{i+1}$, we conclude that
	$\varrho = f$ and the fixed point is unique. 
\end{proof}

\subsection{Proof of Sections~\ref{s:cv_approach}--\ref{s:opt} Results}
\label{ss:pf_main}

When connecting to the results in Section~\ref{ss:aprel}, we always take $\tau_0 = W_0$ and $\tau_1 = W_1$. The map $\tau_2$ and the candidate spaces $E_0$, $E_1$ and $E_2$ will vary depending on the context.

\begin{proof}[Proof of Proposition \ref{p:iter}.]
	\hypertarget{p:iter}
	The claims are immediate from Lemma \ref{l:mof}. Regarding iterations on $S$ and $T$, we set $E_0 = \VV$, $E_1 = \GG$, $E_2 = \HH$ and $\tau_2 = M$. Regarding iterations on $S_{\sigma}$ and $T_{\sigma}$, we fix $\sigma \in \Sigma$ and set $E_0 = \vV$, $E_1 = \gG$, $E_2 = \hH$ and $\tau_2 = M_{\sigma}$. 
\end{proof}

\begin{proof}[Proof of Proposition~\ref{p:exu}.]
	\hypertarget{p:exu}
	The claims are immediate by Lemmas~\ref{l:mof}--\ref{l:exf} (let $E_0 = \VV$, $E_1 = \GG$, $E_2 = \HH$ and $\tau_2 = M$).
\end{proof}

\begin{proof}[Proof of Proposition~\ref{p:vsigi}.]
	\hypertarget{p:vsigi}
	Similar to the proof of Proposition~\ref{p:exu}, this result is immediate
	from Lemmas~\ref{l:mof}--\ref{l:exf} once we set $E_0 = \vV$, $E_1 = \gG$, $E_2 = \hH$ and $\tau_2 = M_{\sigma}$.
\end{proof}

\begin{proof}[Proof of Theorem~\ref{t:bpowm}.]
	\hypertarget{t:bpowm}
	Now suppose that (a) of Theorem~\ref{t:bpowm} holds, so that $g^* \in \GG$ and $Sg^* = g^*$.  We claim that (b) holds, i.e., $v^* \in \VV$ and $T v^* = v^*$. Based on Lemma~\ref{l:mof} (in particular, let $E_0 = \VV$, $E_1 = \GG$, $E_2 = \HH$ and $\tau_2 = M$), we only need to show that $v^* = M W_1 g^*$. 
	
	Pick any $\sigma \in \Sigma$. We have
	\begin{equation*}
	v_\sigma = M_\sigma W_1 g_\sigma 
	\leq M W_1 g_\sigma \leq M W_1 g^*,
	\end{equation*}
	where the equality is due to Assumption~\ref{a:ess} and Proposition~\ref{p:vsigi}, the first inequality is due to the fact that $M$ pointwise dominates $M_\sigma$, and the second follows from the definition of $g^*$ and the monotonicity of $M$ and $W_1$. As $\sigma$ was arbitrary, this proves that $v^* \leq M W_1 g^*$. 
	
	Regarding the reverse inequality, from $g^* = S g^* \in \GG$, there exists a $v \in \VV$ such that $g^* = W_0 v$. Assumption~\ref{a:maxex} then implies the existence of a $\sigma$ such that $M W_1 g^* = M W_1 W_0 v = M_\sigma W_1 W_0 v = M_\sigma W_1 g^*$. Hence, $S g^* = W_0 M W_1 g^* = W_0 M_\sigma W_1 g^* = S_\sigma g^*$. We see that $g^*$ is a fixed point of $S_\sigma$ in $\gG$. Since  $g_\sigma$ is the unique fixed point of $S_\sigma$ in $\gG$ by Assumption~\ref{a:ess} and Proposition~\ref{p:vsigi}, we have $g^* = g_\sigma$. As a result, $M W_1 g^* = M_\sigma W_1 g_\sigma = v_\sigma \leq v^*$. We have now shown that (b) holds. 
	
	The claim that (b) implies (a) is similar. In particular, based on Lemma~\ref{l:mof}, it suffices to show that $g^* = W_0 v^*$. Pick any $\sigma \in \Sigma$. Since $g_\sigma = W_0 v_\sigma \leq W_0 v^*$ by Assumption~\ref{a:ess}, Proposition~\ref{p:vsigi} and the monotonicity of $W_0$, it follows that $g^* \leq W_0 v^*$. Moreover, since $v^* \in \VV$, Assumption~\ref{a:maxex} implies the existence of a $\sigma$ such that $T v^* = T_\sigma v^*$. Since in addition $v^* = T v^*$, it follows that $v^*$ is a fixed point of $T_\sigma$ and thus $v^* = v_\sigma$ by Assumption~\ref{a:ess}. We then have $W_0 v^* = W_0 v_\sigma = g_\sigma \leq g^*$. In summary, we have $g^* = W_0 v^*$ and claim (a) is now verified. 
	
	The claim that $g^* = \hat g$ under (a)--(b) translates to the claim that $g^* = W_0 v^*$, which we have already shown in the previous step.
	
	The claim in Theorem~\ref{t:bpowm} that at least one optimal policy exists
	under either (a) or (b) follows from the equivalence of (a) and (b) just
	established combined with Theorem~\ref{t:bpost}.
	
	Finally, suppose that (a) holds and consider the last claim in
	Theorem~\ref{t:bpowm}, which can be expressed succinctly as
	\begin{equation*}
	\label{eq:ggio}
	v_\sigma = v^* 
	\iff T_\sigma v^* = T v^*
	\iff M_\sigma W_1 g^* = M W_1 g^*.
	\end{equation*}
	Regarding the first equivalence, if $v_\sigma = v^*$, then $T_\sigma v^* = T_\sigma v_\sigma = v_\sigma = v^* = T v^*$. Conversely, if $T_\sigma v^* = T v^*$, then, since $v^* = T v^*$ and, by Assumption~\ref{a:ess}, $v_\sigma$ is the only fixed point of $T_\sigma$ in $\vV$, we have $v_\sigma = v^*$. The second equivalence follows, since, by the relations $g_\sigma = W_0 v_\sigma$ and $g^* = W_0 v^*$, we have: 
	$T_\sigma v^* = M_\sigma W_1 W_0 v^* = M_\sigma W_1 g^*$ and
	$T v^* = M W_1 W_0 v^* = M W_1 g^*$. The last claim in Theorem~\ref{t:bpowm} is now established. This concludes the proof.
\end{proof}

\begin{proof}[Proof of Theorem~\ref{t:asc}.]
	\hypertarget{t:asc}
	We need only show that (a) holds, since the remaining claims follow directly from the hypotheses of the theorem, claim~(a), the definition of asymptotic stability (for part (b)) and Theorem~\ref{t:bpowm} (for parts (c) and (d)).
	
	To see that (a) holds, note that, by the stated
	hypotheses, $S$ has a unique fixed point in $\GG$, which we denote
	below by $\bar g$.  Our aim is to show that $\bar g = g^*$.
	
	First observe that, by Assumption~\ref{a:maxex}, there is a $\sigma
	\in \Sigma$ such that $S_\sigma \bar g = S \bar g = \bar g$.  But,
	by Assumption~\ref{a:ess} and Proposition~\ref{p:vsigi},
	$S_\sigma$ has exactly one fixed point, which is the refactored 
	$\sigma$-value function $g_\sigma$.  Hence $\bar g = g_\sigma$.  
	In particular, $\bar g \leq g^*$.
	
	To see that the reverse inequality holds, pick any $\sigma \in \Sigma$ and
	note that, by the definition of $S$, we have $\bar g = S \bar g \geq S_\sigma \bar g$.  We know that $S_\sigma$ is
	monotone on $\GG$, since this operator is the composition of three
	monotone operators ($W_0$ and $W_1$ by assumption and $M_\sigma$
	automatically).  Hence we can iterate on the last inequality to establish
	$\bar g \geq S^k_\sigma \bar g$ for all $k \in \NN$.
	Taking limits and using the fact that the partial order is closed,
	$S_\sigma$ is asymptotically
	stable and, as shown above, $g_\sigma$ is the unique fixed point,
	we have $\bar g \geq g_\sigma$.  Since $\sigma$ was chosen
	arbitrarily, this yields $\bar g \geq g^*$.  Part (a) of
	Theorem~\ref{t:asc} is now verified.
\end{proof}

\begin{proof}[Proof of Proposition~\ref{p:cc}.]
	\hypertarget{p:cc}
	As a closed subset of $b_\kappa \FF$ under the $\| \cdot
	\|_\kappa$-norm metric, the set $\gG$ is complete under the same metric
	and, by \eqref{eq:ucons} and the Banach contraction mapping theorem, each
	$S_\sigma$ is asymptotically stable on $\gG$. Moreover, the pointwise partial
	order $\leq$ is closed under this metric.  Thus, to verify the conditions
	of Theorem~\ref{t:asc}, we need only show that $S$ is asymptotically
	stable on $\GG$ under the same metric.
	To this end, we first claim that, under the stated assumptions,
	\begin{equation}
	\label{eq:smx}
	Sg (x, a) = \sup_{\sigma \in \Sigma} S_\sigma g (x, a)
	\quad \text{for all } (x, a) \in \FF
	\text{ and } g \in \GG.
	\end{equation}
	To see that \eqref{eq:smx} holds, fix $g \in \GG$. Notice that the definition of $M$ implies that $M \, W_1 \, g \geq M_\sigma \, W_1 \, g$ for any $\sigma \in \Sigma$ and hence, applying $W_0$ to both sides and using monotonicity yield $S g \geq S_\sigma g$ for all $\sigma \in \Sigma$. Moreover, by the definition of $\GG$ and Assumption~\ref{a:maxex}, there exists a $g$-greedy policy $\sigma^* \in \Sigma$ such that $M \, W_1 \, g = M_{\sigma^*} \, W_1 \, g$, and applying $W_0$ to both sides again gives $S g = S_{\sigma^*} g$. Hence \eqref{eq:smx} is valid.
	
	Now fix $g, g' \in \GG$ and $(x, a) \in \FF$.  By \eqref{eq:smx} and the
	contraction condition on $S_\sigma$ in \eqref{eq:ucons}, we have
	\begin{equation*}
	|S g (x, a) - Sg'(x, a)|
	= |\sup_{\sigma \in \Sigma} S_\sigma g (x, a) 
	- \sup_{\sigma \in \Sigma} S_\sigma g'(x, a)|
	\leq \sup_{\sigma \in \Sigma} 
	| S_\sigma g (x, a) - S_\sigma g'(x, a)|,
	\end{equation*}
	\begin{equation*}
	\therefore \qquad
	\frac{|S g (x, a) - Sg'(x, a)|}{ \kappa(x, a)}
	\leq \sup_{\sigma \in \Sigma} \| S_\sigma g - S_\sigma g' \|_\kappa
	\leq \alpha \| g -  g' \|_\kappa.
	\end{equation*}
	Taking supremum gives $\| S g - S g' \|_\kappa \leq \alpha \,  \| g - g' \|_\kappa$, so $S$ is a contraction on $\GG$ of modulus $\alpha$. Since in addition $\GG$ is a closed subset of $b_\kappa \FF$ under the $\|\cdot\|_\kappa$-norm metric, it is a Banach space under the same norm. The Banach contraction mapping theorem then implies that $S$ is asymptotically stable on $\GG$. Since all the conditions of Theorem~\ref{t:asc} are verified, its conclusions follow. 
\end{proof}

\begin{proof}[Proof of Theorem~\ref{t:mpi}.]
	\hypertarget{t:mpi}
	Without loss of generality, we assume that $m_k \equiv m$ for all 
	$k \in \NN_0$. When $k = 0$, by definition, 
	\begin{align*}
	\Sigma_{0}^v = 
	\left\{
	\sigma \in \Sigma \colon 
	M_{\sigma} W_1 W_0 v_{0} = M W_1 W_0 v_{0}
	\right\}    
	= 
	\left\{ 
	\sigma \in \Sigma \colon
	M_{\sigma} W_1 g_{0} = M W_1 g_{0}
	\right\}
	= \Sigma_{0}^g,    
	\end{align*}
	i.e., the sets of $v_{0}$-greedy and $g_0$-greedy policies are identical. This in turn implies that a policy $\sigma_0 \in \Sigma$ is $v_{0}$-greedy if and only if it is 
	$g_{0}$-greedy. Moreover, by Proposition \ref{p:iter}, 
	\begin{align*}
	g_{1}
	&= S_{\sigma_{0}}^{m} g_{0}
	= W_0 T_{\sigma_{0}}^{m - 1} M_{\sigma_0} W_1 g_{0}    \\
	&= W_0 T_{\sigma_{0}}^{m - 1} M_{\sigma_0} W_1 W_0 v_{0}    
	= W_0 T_{\sigma_{0}}^{m - 1} T_{\sigma_{0}} v_{0}
	= W_0 T_{\sigma_{0}}^{m} v_{0}
	= W_0 v_{1}.
	\end{align*}
	The related claims are verified for $k=0$. Suppose these 
	claims hold for arbitrary $k$. It remains to show that they hold for 
	$k+1$. By the induction hypothesis, $\Sigma_k^v = \Sigma_k^g$, 
	a policy $\sigma_k \in \Sigma$ is $v_{k}$-greedy if and only if it is 
	$g_{k}$-greedy, and $g_{k+1} = W_0 v_{k+1}$.
	By definition of the greedy policy,
	\begin{align*}
	\Sigma_{k+1}^v &= 
	\left\{ 
	\sigma \in \Sigma \colon
	M_{\sigma} W_1 W_0 v_{k+1}
	= M W_1 W_0 v_{k+1}
	\right\}    \\
	&= 
	\left\{
	\sigma \in \Sigma \colon
	M_{\sigma} W_1 g_{k+1}
	= M W_1 g_{k+1}
	\right\}
	= \Sigma_{k+1}^g,  
	\end{align*}
	i.e., a policy $\sigma_{k+1} \in \Sigma$ is $v_{k+1}$-greedy if and only if it is $g_{k+1}$-greedy. Moreover, by Proposition \ref{p:iter},
	\begin{align*}
	g_{k+2}
	&= S_{\sigma_{k+1}}^{m} g_{k+1}
	= W_0 T_{\sigma_{k+1}}^{m - 1} 
	M_{\sigma_{k+1}} W_1 g_{k+1}    \\
	&= W_0 T_{\sigma_{k+1}}^{m - 1} M_{\sigma_{k+1}} 
	W_1 W_0 v_{k+1}    
	= W_0 T_{\sigma_{k+1}}^{m} v_{k+1}
	= W_0 v_{k+2}.
	\end{align*}
	Hence, the related claims of the proposition hold for $k+1$, completing the
	proof.
\end{proof}

\begin{proof}[Proof of Theorem \ref{t:hpi_conv}.]
	\hypertarget{t:hpi_conv}
	Without loss of generality, we assume $m_k \equiv m$ for all $k \in \NN_0$. Let $\{ \zeta_k \}$ be defined by $\zeta_0 := g_{0}$ and $\zeta_k := S \, \zeta_{k-1}$. We show by induction that
	\begin{align}
	\label{eq:ineq_mpi}
	\zeta_k \leq g_{k} \leq g^*
	\quad \text{and} \quad
	g_{k} \leq S g_{k}
	\quad \text{for all } \, k \in \NN_0.
	\end{align} 
	Note that $g_{0} \leq S g_{0}$ by assumption. Since $S$ is monotone, this implies that $g_{0} \leq S^t g_{0}$ for all $t \in \NN_0$. Letting $t \to \infty$, Theorem \ref{t:asc} implies that $g_{0} \leq g^*$.
	Since in addition  $\zeta_0 = g_{0}$, \eqref{eq:ineq_mpi} is 
	satisfied for $k=0$. Suppose this claim holds for arbitrary $k \in \NN_0$. 
	Then, since $S$ and $S_{\sigma_{k}}$ are monotone, 
	\begin{align*}
	g_{k+1}
	= S_{\sigma_{k}}^{m} \, g_{k} 
	\leq S S_{\sigma_{k}}^{m-1} g_{k}
	\leq S S_{\sigma_{k}}^{m-1} S g_{k}
	= S S_{\sigma_{k}}^{m} \, g_{k}
	= S g_{k+1},
	\end{align*}
	where the first inequality holds by the definition of $S$, the second 
	inequality holds since $g_{k} \leq S g_{k}$ (the induction hypothesis), and the second equality is due to $\sigma_{k} \in \Sigma_k^g$. Hence, 
	$g_{k+1} \leq S^t g_{k+1}$ for all $t \in \NN_0$. Letting $t \to \infty$ yields $g_{k+1} \leq g^*$.
	
	Similarly, since 
	$g_{k} \leq S g_{k} = S_{\sigma_{k}} g_{k}$, we have
	$g_{k} \leq S_{\sigma_{k}}^{m-1} g_{k}$ and thus
	\begin{align*}
	\zeta_{k+1} = S \zeta_k \leq S g_{k}
	= S_{\sigma_{k}} g_{k}
	\leq S_{\sigma_{k}}^m g_{k} = g_{k+1}.
	\end{align*}
	Hence, \eqref{eq:ineq_mpi} holds by induction. The monotonicity of 
	$S_{\sigma_{k}}$ implies that
	%
	%    \begin{align*}
	%      g_{k}
	%      \leq S g_{k} 
	%      = S_{\sigma_{k}} g_{k}
	%      \leq \cdots 
	%      \leq S_{\sigma_{k}}^m g_{k}
	%      = g_{k+1}
	%    \end{align*}
	%
	$g_{k}
	\leq S g_{k} 
	= S_{\sigma_{k}} g_{k}
	\leq \cdots 
	\leq S_{\sigma_{k}}^m g_{k}
	= g_{k+1}$
	for all $k \in \NN_0$. Hence, $\{ g_{k}\}$ is increasing in $k$. 
	Moreover, since $\zeta_k \to g^*$ by Theorem \ref{t:asc}, \eqref{eq:ineq_mpi} implies that $g_{k} \to g^*$ as $k \to \infty$.
\end{proof}

\subsection{Proof of Section~\ref{ss:os} Results}
\label{ss:pf_cmpl}

To prove the results of Table \ref{tb:tm_cmplx}, recall that floating point operations are any elementary actions (e.g., $+$, $\times$, $\max$, $\min$) on or assignments with floating point numbers. If $f$ and $h$ are scalar functions on $\RR^n$, we write $f(x)= \mathcal{O}(h(x))$ whenever there exist $C, M>0$ such that $\| x \| \geq M$ implies $|f(x)| \leq C |h(x)|$, where $\|\cdot \|$ is the sup norm. 
We introduce some elementary facts on time complexity:
\begin{enumerate}
	\item[(a)] The multiplication of an $m \times n$ matrix 
	and an $n \times p$ matrix has complexity $\mathcal{O}(mnp)$. 
	(See, e.g., Section~2.5.4 of \cite{skiena2008algorithm}.)
	
	\item[(b)] The  
	binary search algorithm finds the index of an element in a given 
	sorted array of length $n$ in $\mathcal{O}(\log(n))$ time. 
	(See, e.g., Section~4.9 of \cite{skiena2008algorithm}.)
\end{enumerate}

For the finite state (FS) case, time complexity of $1$-loop VFI reduces to
the complexity of matrix multiplication $\hat{P} \vec{v}$, which is of order 
$\mathcal{O}(L^2 K^2)$ based on the shape of $\hat{P}$ and $\vec{v}$ and 
fact~(a) above. Similarly, time complexity of $1$-loop RVFI has complexity $\mathcal{O}(L K^2)$. 
Time complexity of $n$-loop algorithms is scaled up by $\mathcal{O}(n)$. 

For the continuous state (CS) case, let $\mathcal{O}(\phi)$ and 
$\mathcal{O}(\psi)$ denote respectively the complexity of single point evaluation of 
the interpolating functions $\phi$ and $\psi$. Counting the 
floating point operations associated with all grid points inside the 
inner loops shows that the one step complexities 
of VFI and RVFI are $\mathcal{O}(NLK) \mathcal{O}(\psi)$ and
$\mathcal{O}(NK) \mathcal{O}(\phi)$, respectively. 
Since binary search function evaluation is adopted for the indicated 
function interpolation mechanisms, and in particular, since evaluating 
$\psi$ at a given point uses binary search $\ell + k$ times, based on fact~(b) above, we have
\begin{align*}
\mathcal{O}(\psi) 
&= \mathcal{O} \left( \sum_{i=1}^{\ell} \log(L_i) +
\sum_{j=1}^{k} \log(K_j) 
\right) \\
&= \mathcal{O} \left( \log \left( \Pi_{i=1}^\ell L_i \right) 
+ \log \left( \Pi_{j=1}^k K_j \right)
\right)
= \mathcal{O}(\log(LK)).
\end{align*}
Similarly, we can show that $\mathcal{O}(\phi) = \mathcal{O}(\log(K))$. Combining 
these results, we see that the claims of the CS case hold, concluding our proof
of Table \ref{tb:tm_cmplx} results.

\subsection{Counterexamples}
\label{ss:counter}

In this section, we provide counterexamples showing that monotonicity of $W_0$
and $W_1$ cannot be dropped from Theorem~\ref{t:bpowm}. Recall that $\EE_{y \mid x}$ denotes 
the expectation of $y$ conditional on $x$.

\subsubsection{Counterexample 1}
\label{sss:ce1}

Here we exhibit a dynamic program and value transformation under which $Tv^* = v^*$, $S g^* \not= g^*$ and $g^* \neq \hat g$. The example involves risk sensitive preferences.

Let $\XX := \{ 1, 2 \}$, $\AA := \{ 0, 1\}$, $\vV = \VV := \RR^\XX$ and $\Sigma :=
\{ \sigma_1, \sigma_2, \sigma_3, \sigma_4\}$, where the policy functions are
%
%\vspace*{-0.15cm}
\begin{center}
	{\small
		\begin{tabular}{|c|c|c|c|c|}
			\hline 
			$\;\;x\;\;$ & $\sigma_{1}(x)$ & $\sigma_{2}(x)$ & $\sigma_{3}(x)$ & $\sigma_{4}(x)$\tabularnewline
			\hline 
			\hline 
			$1$ & $0$ & $0$ & $1$ & $1$\tabularnewline
			\hline 
			$2$ & $0$ & $1$ & $1$ & $0$\tabularnewline
			\hline 
		\end{tabular}
	}  
	\par\end{center}
%\vspace*{0.35cm}
%
In state $x$, choosing action $a$ gives the agent an immediate reward
$x - a$. If $a = 0$, then the next period state $x'=1$, while if $a=1$, the
next period state $x' = 2$. Let the state-action aggregator $H$ take the
risk-sensitive form
\begin{equation*}
H(x,a,v) = x - a - \frac{\beta}{\gamma} \log \EE_{x'|a} \exp[-\gamma v(x')].
\end{equation*} 
Then the Bellman operator $T$ is 
\begin{align*}
T v (x) 
= \max_{a \in \{ 0,1\}}
\left\{ 
x - a - \frac{\beta}{\gamma} \log \EE_{x'|a} \exp [- \gamma v(x')]
\right\}    
= \max \left\{ x + \beta v(1), \, x - 1 + \beta v(2) \right\}.
\end{align*}
Consider the plan factorization $(W_0, W_1)$ given by
\begin{equation*}
W_0 \, v (a) := \EE_{x'|a} \exp [- \gamma v(x')] 
\quad \text{and} \quad
W_1 \,  g(x, a) := x - a - \frac{\beta}{\gamma} \log g(a).
\end{equation*}
The refactored Bellman operator is then
\begin{equation*}
S g(a) = \EE_{x'|a} \exp 
\left( - \gamma
\max_{a' \in \{ 0,1 \}}
\left\{
x' - a' - \frac{\beta}{\gamma} \log g(a')
\right\}
\right).
\end{equation*}
Note that neither $W_1$ nor $W_0$ is monotone. In the following, we assume that
$\beta \in (0,1)$.

\begin{lemma}
	\label{lm:s0_sigv_ce1}
	The $\sigma$-value functions are as follows:
	%    %
	%    \begin{align*}
	%        & v_{\sigma_1} (1) = 1 / (1 - \beta),
	%        \quad 
	%        v_{\sigma_1} (2) = (2 - \beta) / (1 - \beta),    
	%        \quad
	%        v_{\sigma_2} (1) = 1 / (1 - \beta),
	%        \quad 
	%        v_{\sigma_2} (2) = 1 / (1 - \beta),    \\
	%        & v_{\sigma_3} (1) = \beta / (1 - \beta),
	%        \quad 
	%        v_{\sigma_3} (2) = 1 / (1 - \beta),   
	%        \quad 
	%        v_{\sigma_4} (1) = (2 \beta) / (1 - \beta^2),
	%        \quad 
	%        v_{\sigma_4} (2) = 2 / (1 - \beta^2).
	%    \end{align*}
	%    %
	%\vspace*{-0.15cm}
	\begin{center}
		{\small
			\begin{tabular}{|c|c|c|c|c|}
				\hline 
				$\;\;x\;\;$ & $v_{\sigma_{1}}(x)$ & $v_{\sigma_{2}}(x)$ & $v_{\sigma_{3}}(x)$ & $v_{\sigma_{4}}(x)$\tabularnewline
				\hline 
				\hline 
				$1$ & $1/(1-\beta)$ & $1/(1-\beta)$ & $\beta/(1-\beta)$ & $2\beta/(1-\beta^{2})$\tabularnewline
				\hline 
				$2$ & $(2-\beta)/(1-\beta)$ & $1/(1-\beta)$ & $1/(1-\beta)$ & $2/(1-\beta^{2})$\tabularnewline
				\hline 
			\end{tabular}
		}
		\par\end{center}
	%\vspace*{0.3cm}
	%
\end{lemma}

\begin{proof}
	Regarding $v_{\sigma_1}$, by definition,
	\begin{align*}
	&v_{\sigma_1} (1) 
	= 1 - \sigma_1(1) - \frac{\beta}{\gamma} \log \EE_{x'|\sigma_1(1)} 
	\exp [- \gamma v_{\sigma_1} (x')]
	= 1 + \beta v_{\sigma_1} (1)    \\
	\text{and} \quad
	&v_{\sigma_1} (2)
	= 2 - \sigma_1(2) - \frac{\beta}{\gamma} \log \EE_{x'|\sigma_1(2)} 
	\exp [- \gamma v_{\sigma_1} (x')]
	= 2 + \beta v_{\sigma_1} (1).
	\end{align*}
	Hence, $v_{\sigma_1} (1) = 1 / (1 - \beta)$ and 
	$v_{\sigma_1}(2) = (2- \beta) / (1 - \beta)$. 
	The remaining proofs are similar. 
\end{proof}

Based on Lemma \ref{lm:s0_sigv_ce1}, the value function $v^*$ is given by
\begin{align*}
&v^* (1) = \max_{\sigma_i} v_{\sigma_i} (1) 
= v_{\sigma_1} (1) 
= v_{\sigma_2} (1) 
= 1 / (1 - \beta)    
\\
\text{and} \quad
&v^* (2) = \max_{\sigma_i} v_{\sigma_i} (2) 
= v_{\sigma_1} (2)
= (2 - \beta) / (1 - \beta).
\end{align*}
Moreover, we have $v^* = T v^*$, since
\begin{align*}
T v^*(1) 
&= \max \{ 1 + \beta v^*(1), \, \beta v^*(2) \} 
= 1 / (1 - \beta) = v^*(1)    \\
\text{and} \quad
T v^*(2) 
&= \max
\{ 2 + \beta v^*(1), \, 1 + \beta v^*(2) \} 
= (2 - \beta) / (1 - \beta) = v^* (2).
\end{align*}

\begin{lemma}
	\label{lm:s1_sigv_ce1}
	The refactored $\sigma$-value functions are as follows:
	%    %
	%    \begin{align*}
	%        & g_{\sigma_1} (0) = \exp [ -\gamma / (1 - \beta)], \quad
	%         g_{\sigma_1} (1) = \exp [ -\gamma (2 - \beta)/(1 - \beta) ],     \\
	%        & g_{\sigma_2} (0) = \exp[-\gamma / (1 - \beta)], \quad
	%         g_{\sigma_2} (1) = \exp[-\gamma / (1 - \beta)],    \\
	%        & g_{\sigma_3} (0) = \exp[- \gamma \beta / (1 - \beta)], \quad
	%         g_{\sigma_3} (1) = \exp[- \gamma / (1 - \beta)],    \\
	%        & g_{\sigma_4} (0) = \exp[ -2 \gamma \beta / (1 - \beta^2)], \quad
	%         g_{\sigma_4} (1) = \exp[- 2 \gamma / (1 - \beta^2)].
	%    \end{align*}
	%    %
	%\vspace*{0.35cm}
	\begin{center}
		{\small
			\begin{tabular}{|c|c|c|c|c|}
				\hline 
				$\;\;a\;\;$ & $g_{\sigma_{1}}(a)$ & $g_{\sigma_{2}}(a)$ & $g_{\sigma_{3}}(a)$ & $g_{\sigma_{4}}(a)$\tabularnewline
				\hline 
				\hline 
				$0$ & $\exp[-\gamma/(1-\beta)]$ & $\exp[-\gamma/(1-\beta)]$ & $\exp[-\gamma\beta/(1-\beta)]$ & $\exp[-2\gamma\beta/(1-\beta^{2})]$\tabularnewline
				\hline 
				$1$ & $\exp[-\gamma(2-\beta)/(1-\beta)]$ & $\exp[-\gamma/(1-\beta)]$ & $\exp[-\gamma/(1-\beta)]$ & $\exp[-2\gamma/(1-\beta^{2})]$\tabularnewline
				\hline 
			\end{tabular}
		}
		\par\end{center}
\end{lemma}

\begin{proof}
	Regarding $g_{\sigma_1}$, by definition,
	\begin{align*}
	g_{\sigma_1} (0) 
	&= \EE_{x'|0} \exp 
	\left( - \gamma
	\left\{ 
	x' - \sigma_1 (x') - 
	\frac{\beta}{\gamma} \log g_{\sigma_1} [\sigma_1(x')]
	\right\}
	\right)    \\
	&= \exp 
	\left( - \gamma
	\left[ 
	1 - 
	\frac{\beta}{\gamma} \log g_{\sigma_1} (0)
	\right]
	\right)  
	= \exp (-\gamma) g_{\sigma_1} (0)^{\beta}
	\end{align*}
	\begin{align*}
	\text{and} \quad
	g_{\sigma_1} (1)
	&=\EE_{x'|1} \exp 
	\left( - \gamma
	\left\{ 
	x' - \sigma_1 (x') - 
	\frac{\beta}{\gamma} \log g_{\sigma_1} [\sigma_1(x')]
	\right\}
	\right)    \\
	&= \exp 
	\left( - \gamma
	\left[ 
	2 - 
	\frac{\beta}{\gamma} \log g_{\sigma_1} (0)
	\right]
	\right)  
	= \exp(-2 \gamma) g_{\sigma_1} (0)^{\beta}.
	\end{align*}
	So $g_{\sigma_1}(0)$ and $g_{\sigma_1}(1)$ are as shown in the table. 
	%$g_{\sigma_1} (0) = \exp [ -\gamma / (1 - \beta)]$ and
	%$g_{\sigma_1} (1) = \exp[ - \gamma (2 - \beta) / (1 - \beta)]$.
	The remaining proofs are similar.
\end{proof}

Based on Lemma \ref{lm:s1_sigv_ce1}, the refactored value function $g^*$ satisfies
\begin{align*}
&g^*(0) = \max_{\sigma_i} g_{\sigma_i} (0)
= g_{\sigma_3} (0) = \exp[ - \gamma \beta / (1 - \beta)]    \\
\text{and} \quad
&g^*(1) = \max_{\sigma_i} g_{\sigma_i}(0)
= g_{\sigma_2} (1) = g_{\sigma_3} (1) = \exp [ - \gamma / (1 - \beta)].
\end{align*}
Since $\hat{g}(0) = W_0 v^*(0) = \exp[-\gamma v^*(1)] = \exp[-\gamma / (1 - \beta)] \neq g^*(0)$, we see that $g^* \neq \hat{g}$. Moreover,
\begin{align*}
S g^*(0) 
%    &= \EE_{0} \exp 
%    \left( - \gamma
%        \max_{a' \in \{0,1\}} 
%        \left\{ x' - a' - \frac{\beta}{\gamma} \ln g^* (a') \right\}    
%    \right)
%        \\
&= \exp 
\left( - \gamma
\max
\left\{ 
1 - \frac{\beta}{\gamma} \log g^* (0),
- \frac{\beta}{\gamma} \log g^*(1)
\right\}    
\right)    \\ 
&= \exp [- \gamma (1 - \beta + \beta^2) / (1 - \beta^2) ]
\neq g^*(0).
\end{align*}
Hence, $g^* \neq S g^*$. The refactored  value function is not a fixed
point of the refactored  Bellman operator.

\subsubsection{Counterexample 2}

We now exhibit a dynamic program and plan factorization under which 
$S g^* = g^*$, $T v^* \neq v^*$ and $g^* \neq \hat{g}$.

The set up is same as Section \ref{sss:ce1}, except that we let $\beta > 1$. In this case, Lemmas 
\ref{lm:s0_sigv_ce1}--\ref{lm:s1_sigv_ce1} still hold and, as in 
Section~\ref{sss:ce1}, the value function and refactored value function satisfy
\begin{align*}
&v^*(1) = 1 / (1 - \beta), \quad 
v^*(2) = (2 - \beta) / (1 - \beta), \;   \\
&g^*(0) = \exp[- \gamma \beta / (1 - \beta)] \quad \text{and} \quad
g^*(1) = \exp[- \gamma / (1 - \beta)].
\end{align*}
We have seen in Section \ref{sss:ce1} that $g^* \neq \hat{g}$. Since in addition  
\begin{align*}
S g^*(0) 
%	&= \EE_{0} \exp 
%	\left( - \gamma
%	\max_{a' \in \{0,1\}} 
%	\left\{ x' - a' - \frac{\beta}{\gamma} \ln g^* (a') \right\}    
%	\right)
%	\\
&= \exp 
\left( - \gamma
\max
\left\{ 
1 - \frac{\beta}{\gamma} \log g^* (0),
- \frac{\beta}{\gamma} \log g^*(1)
\right\}    
\right)     \\
&= \exp [- \gamma \beta / (1 - \beta) ]
= g^*(0)    
\end{align*}
\begin{align*}
\text{and} \quad
S g^*(1) 
%	&= \EE_{1} \exp 
%	\left( - \gamma
%	\max_{a' \in \{0,1\}} 
%	\left\{ x' - a' - \frac{\beta}{\gamma} \ln g^* (a') \right\}    
%	\right)
%	\\
&= \exp 
\left( - \gamma
\max
\left\{ 
2 - \frac{\beta}{\gamma} \log g^* (0),
1 - \frac{\beta}{\gamma} \log g^*(1)
\right\}    
\right)  \\     
&= \exp [- \gamma / (1 - \beta) ]
= g^*(1),
\end{align*}
we have $S g^* = g^*$ as claimed. However, since
\begin{align*}
T v^*(1) &= \max 
\left\{ 
1 - \frac{\beta}{\gamma} \log \EE_{x'|0} \exp [- \gamma v^* (x')], \,
- \frac{\beta}{\gamma} \log \EE_{x'|1} 
\exp [- \gamma v^* (x')]
\right\}    \\
&= \max \{ 1 + \beta v^*(1), \, \beta v^*(2) \} 
= (2 \beta - \beta^2) / (1 - \beta) \neq v^*(1), 
\end{align*}
we have $T v^* \neq v^*$, so the value function is not a fixed point of the Bellman operator.

\bibliographystyle{ecta}

\bibliography{dpd}

\end{document}